\documentclass[reqno]{amsart}
\usepackage{enumerate}
\usepackage{mhequ}
\usepackage[margin=1.2in]{geometry}
\usepackage{amssymb,url,color, booktabs,nccmath}
\usepackage{mathrsfs}
\usepackage{enumitem}
\usepackage{graphicx}
\usepackage{tikz}
\usetikzlibrary{decorations.shapes}
\usepackage{color}
\usepackage[colorlinks=true]{hyperref}
\hypersetup{
    linkcolor=blue,          % color of internal links
    citecolor=red,        % color of links to bibliography
    filecolor=blue,      % color of file links
    urlcolor=cyan
}

\numberwithin{equation}{section}

\newtheorem{theorem}{Theorem}[section]

\newtheorem{proposition}[theorem]{Proposition}

\newtheorem{assumptioninner}{Assumption}

\newenvironment{assumption}
  {\begin{assumptioninner}}
  {\end{assumptioninner}}
\newenvironment{nouppercase}{%
  \renewcommand{\uppercasenonmath}[1]{}}{}
\def\dd{\,\mathrm{d}}
\def\geq{\geqslant}
\def\ge{\geqslant}
\def\leq{\leqslant}
\def\le{\leqslant}

\begin{document}

\title[Learning Equilibrium Expansions from Dynamics]{\LARGE Learning Equilibrium Fluctuation Expansions from Overdamped Langevin Dynamics}

\author[Lin Wang]{\large Lin Wang}
\address[L. Wang]{School of Mathematical Sciences, University of Chinese Academy of Sciences, Beijing 100049, China; Academy of Mathematics and Systems Science, Chinese Academy of Sciences, Beijing 100190, China.}
\email{wanglin2021@amss.ac.cn}

\author[Zhengyan Wu]{\large Zhengyan Wu}
\address[Z. Wu]{Department of Mathematics, Technische Universit\"at M\"unchen, Boltzmannstr. 3, 85748 Garching, Germany}
\email{wuzh@cit.tum.de}

\begin{abstract}
	We study higher-order small-noise fluctuation expansions for the overdamped Langevin dynamics in a quartic double-well potential. Assuming that the initial data admits a suitable expansion structure, we obtain a strong dynamical expansion of the trajectories, as well as an expansion of the laws with respect to smooth observables. We then investigate the long-time behavior of the expansion coefficients. In the scalar case $d=1$, each coefficient converges exponentially fast to a finite limit as $t\to\infty$. In contrast, for $d\ge 2$, the fluctuation expansion coefficients reflect the degeneracy of the manifold of minima, which in general prevents the existence of a finite long-time limit. Furthermore, by combining a multi-level induction with combinatorial arguments, we derive a recursive formula for the fluctuation expansion coefficients. This recursion shows that the long-time limits of these dynamical expansion coefficients coincide with those arising from the corresponding equilibrium expansions.

\end{abstract}


\begin{nouppercase}
\maketitle
\end{nouppercase}

\setcounter{tocdepth}{1}
\tableofcontents

\section{Introduction}\label{sec-intro}
In both probability and statistics, the central limit theorem (CLT) plays a central role in characterizing fluctuation phenomena. Going beyond the CLT, higher-order fluctuation expansions provide a more accurate description of the associated approximation errors. Such refinements are commonly known as Edgeworth expansions and have become fundamental tools in probabilistic and statistical theory. In the study of stochastic differential equations (SDEs), under suitable scaling regimes, central limit fluctuations are typically available; it is therefore natural to investigate finer fluctuation structures beyond the Gaussian approximation.

In many situations, asymptotic expansions of the form 
\begin{equation}\label{invariantmeas-expansion}
\int F(x)\,\mu_\varepsilon(\mathrm{d} x) = B_0(F) + \varepsilon^{1/2} B_1(F) + \dots + \varepsilon^{n/2} B_n(F) + o(\varepsilon^{n/2}), \quad \varepsilon\to 0,
\end{equation}
for observables $F\in C^{n+1}$ with respect to a Gibbs-type equilibrium measure
$$
\mu_{\varepsilon}(\mathrm{d}x)=Z_{\varepsilon}^{-1} e^{-V(x)/\varepsilon}\dd x,
$$
attract significant interest. However, these measures are often difficult to compute explicitly. A common strategy, reminiscent of Monte Carlo methods, is to approximate such equilibrium expectations through the trajectories of an associated stochastic dynamics, for instance the overdamped Langevin dynamics with $\mu_\varepsilon$ as its invariant measure. From this perspective, understanding the fluctuation structure of the dynamics provides a way to ``learn'' refined information about equilibrium statistics from dynamical observations.

In this paper, we establish higher-order small-noise fluctuation expansions for the overdamped Langevin dynamics and analyze the long-time behavior of the corresponding expansion coefficients. In particular, the long-time asymptotics reveal a qualitative distinction between systems with nondegenerate minima and those whose potential exhibits degenerate manifolds of minima. Specifically, we consider
\begin{equation}\label{SDE-0}
\dd X_{\varepsilon}(t)=-\nabla V (X_{\varepsilon}(t))\,\dd t+\sqrt{2\varepsilon}\,\dd W_t, \quad X_{\varepsilon}(0)=\xi_\varepsilon,
\end{equation}
where $(W_t)_{t \ge 0}$ is a standard $d$-dimensional Brownian motion and the initial data $\xi_\varepsilon$ is an $\mathbb{R}^d$-valued measurable random variable independent of the Brownian motion $W$. 

For every $n\in \mathbb{N}$, under a matching asymptotic expansion assumption on $\xi_\varepsilon$, we establish a dynamical asymptotic expansion of the form
\begin{equation}\label{eq-dyna-exp}
X_{\varepsilon}=\bar{X}_0+\varepsilon^{1/2}\bar{X}_1+\varepsilon\,\bar{X}_2+\cdots+\varepsilon^{n/2} \bar{X}_n + o(\varepsilon^{n/2}), \quad \text{as } \varepsilon\to 0,
\end{equation}
where the fluctuation coefficients $\bar{X}_m$, $m=1,\dots,n$, are Wiener functionals determined by recursive equations introduced later. 

Moreover, based on \eqref{eq-dyna-exp}, we obtain an asymptotic expansion in the sense of distributions: for every $n\in \mathbb{N}$, $F\in C^{n+1}(\mathbb{R}^d)$, and $t>0$,
\begin{align}\label{eq-weak-exp-intro}
\int_{\mathbb{R}^d} F(x)\,\mathcal{L}_{X_{\varepsilon}(t)}(\mathrm{d}x)
= a_0(t,F) + \varepsilon^{1/2} a_1(t,F) + \cdots + \varepsilon^{n/2} a_n(t,F) + o(\varepsilon^{n/2}), \quad \varepsilon\to 0,
\end{align}
where $\mathcal{L}_{X_{\varepsilon}(t)}$ denotes the law of $X_{\varepsilon}(t)$. In the context of probability theory, expansion \eqref{eq-weak-exp-intro} goes beyond central limit fluctuations, providing control over errors at arbitrarily high orders. 

Building on the analysis of \eqref{invariantmeas-expansion} and motivated by interests in probabilistic sampling, it is natural to investigate the long-time behavior ($t\to\infty$) of the fluctuation sequences $(a_m(t,F))_{1\le m \le n}$ and to understand their relationship with the equilibrium fluctuations $(B_m(F))_{1\le m \le n}$ in \eqref{invariantmeas-expansion}. Furthermore, the scalar case $d=1$ and the vector case $d\ge 2$ correspond to nondegenerate and degenerate geometric structures of the potential, respectively, leading to qualitatively different dynamics. This naturally raises the question of how such structural differences influence the long-time behavior of the fluctuation expansions. To explore these aspects in greater detail, we propose the following concrete questions.

\begin{description}
    \item[Q1] Do the long-time limits ($t \to \infty$) of the fluctuation sequences $(a_m(t,F))$ exist? If so, what are the convergence rates? In particular, do these limits coincide with the equilibrium fluctuations; that is, does
$$
\lim_{t\to\infty} a_m(t,F) = B_m(F) \quad \text{for every } 1\le m \le n \text{ and every } F \in C^{n+1}(\mathbb{R}^d)\,?
$$

    \item[Q2] Can the sequences $(a_m(t,F))$ and $(B_m(F))$ be characterized explicitly via recursive formulas?

    \item[Q3] Do the answers to \textbf{Q1} and \textbf{Q2} differ between the scalar case $d=1$ and the vector case $d \ge 2$?
\end{description}
As a supplement to these questions, we emphasize that \textbf{Q1} is challenging. Heuristically, suppose that the long-time limits of $(a_m(t,F))_{1\le m\le n}$ exist, and denote them by $b_m(F) := \lim_{t \to \infty} a_m(t,F)$ for each $1\le m\le n$. The two families $(b_m(F))$ and $(B_m(F))$ correspond to different orders of limits:
\begin{align*}
    \text{Dynamical: } &\lim_{t\to\infty}\lim_{\varepsilon\to0} \frac{\mathbb{E}[F(X_\varepsilon(t))]-\sum_{k=0}^{m-1}\varepsilon^{k/2}a_k(t,F)}{\varepsilon^{m/2}} = b_m(F),\\
    \text{Equilibrium: } &\lim_{\varepsilon\to0}\lim_{t\to\infty} \frac{\mathbb{E}[F(X_\varepsilon(t))]-\sum_{k=0}^{m-1}\varepsilon^{k/2}B_k(F)}{\varepsilon^{m/2}} = B_m(F).
\end{align*}
In general, it is not clear whether these limits can be interchanged; this depends on the intricate long-term dynamics of $X_\varepsilon$ and the interplay between transition rates and the order of fluctuations. To address these challenges, we develop a careful multi-level induction method based on combinatorial arguments, which allows us to systematically answer \textbf{Q1}, \textbf{Q2}, and \textbf{Q3}.

\subsection{Statement of the main results}
 We begin by introducing some notation. Throughout the paper,  we denote by $\mathbb{N} = \{0, 1, 2, \dots\}$ the set of natural numbers and by $\mathbb{N}_+ = \{1, 2, \dots\}$ the set of positive integers. For $a, b \in \mathbb{R}$, we write $a\asymp b$ if $cb \le a \le Cb$ for some constants $c, C > 0$, and adopt the convention that all such generic constants may change from line to line.
For $x,y\in\mathbb{R}^d$, we denote by $x\cdot y$ their Euclidean inner product and by $|x|$ the associated norm.

For each $n\in \mathbb{N}_+$, let $C^n(\mathbb{R}^d)$ denote the space of real-valued functions on $\mathbb{R}^d$ with continuous derivatives up to order $n$.
For every $F\in C^n(\mathbb{R}^d)$, $x\in\mathbb{R}^d$, and $1\le i\le n$, the $i$-th derivative of $F$ at $x$ is a symmetric $i$-linear map acting on $i$ vectors $v_1, \cdots, v_i\in \mathbb{R}^d$,  defined by 
$$D^iF(x)(v_1,\dots,v_i)
:=\sum_{j_1,\dots,j_i=1}^d
\frac{\partial^i F}{\partial x_{j_1}\dots \partial x_{j_i}}(x)\,(v_1)_{j_1}\dots (v_i)_{j_i}.
$$
When all $i$ vectors are identical, we simply write
\begin{equation*}
D^iF(x)[v]^i:=D^iF(x)(v,\cdots,v).
\end{equation*}
The operator norm of $D^i F(x)$ is defined by
\begin{equation*}
\|D^i F(x)\| := \sup_{v_1, \dots, v_i \in \mathbb{R}^d \setminus \{0\}} \frac{|D^i F(x)(v_1, \dots, v_i)|}{|v_1|\cdots |v_i|}.
\end{equation*}
In the scalar-valued case $d=1$, we have $D^i F(x)=F^{(i)}(x).$

For every $n\in\mathbb{N}_+$ and $i\in\{1,2,\cdots,n\}$, let 
\begin{equation}\label{D-ni}
\mathcal{D}_i^{n}:=\{(j_1,\cdots,j_i)\in(\mathbb{N}_+)^i:j_1+\cdots+j_i=n\}
\end{equation}
 denote the set of all $i$-tuples of positive integers summing to $n$. 
 We adopt the convention that $\mathcal{D}_i^{m}=\varnothing$ whenever $i>m$ or $i\leq 0$ or $m\leq 0$, and that any sum over the empty set is $0$.
Our first main result is a higher-order dynamical fluctuation expansion for $X_\varepsilon$ in the small-noise regime $\varepsilon\to0$.
We assume that the initial data $\xi_\varepsilon$ admits an asymptotic expansion of the form
$$
\xi_\varepsilon = \xi_0 + \varepsilon^{1/2}\xi_1 + \varepsilon\xi_2 + \dots + \varepsilon^{n/2}\xi_n + o\bigl(\varepsilon^{n/2}\bigr),
$$ in the following sense.
\begin{assumption}\label{A1}
    There exist random variables $\xi_0, \xi_1, \dots, \xi_n$ such that, for every $p \ge 1$ and $0 \le k \le n$:
\begin{equation*}
\mathbb{E}|\xi_k|^{2p} < \infty, \quad
\mathbb{E}\left|\xi_\varepsilon - \sum_{j=0}^k \varepsilon^{j/2}\xi_j \right|^{2p} \le C(k,p)\varepsilon^{(k+1)p}.
\end{equation*}
\end{assumption}

The expansion coefficients $\bar X_k$ are determined recursively by substituting \eqref{eq-dyna-exp}
into  \eqref{SDE-0} and matching powers of $\varepsilon^{1/2}$. The leading term $\bar X_0$ solves the deterministic ODE with random initial data:
\begin{equation}\label{eq-barX0-intro}
	\dd\bar X_0(t)=\bigl(\bar X_0(t)-|\bar X_0(t)|^2\bar X_0(t)\bigr)\dd t,\qquad 
	\bar X_0(0)=\xi_0.
\end{equation}
The first-order fluctuation $\bar X_1$ satisfies a linear SDE driven by the Brownian motion $W_t$:
\begin{equation}\label{eq-barX1-intro}
	\dd\bar X_1(t)=\bigl[(1-|\bar X_0(t)|^2)I-2\bar X_0(t)\otimes\bar X_0(t)\bigr]\bar X_1(t)\dd t
	+\sqrt{2}\dd W_t,\qquad \bar X_1(0)=\xi_1.
\end{equation}
For every $m\ge2$, the coefficient $\bar X_m$ solves a linear inhomogeneous equation whose forcing term collects all contributions of total order $m$ from the nonlinearity $|x|^2x$:
\begin{equation}\label{eq-barXm-intro}
	\begin{aligned}
		\dd\bar X_m(t) &=\bigl[(1-|\bar X_0(t)|^2)I-2\bar X_0(t)\otimes\bar X_0(t)\bigr]\bar X_m(t)\dd t \\
		&\quad -\sum_{(i,j)\in\mathcal{D}_2^{m}}
		\Bigl[(\bar X_i(t)\cdot\bar X_j(t))\bar X_0(t)+2(\bar X_i(t)\cdot\bar X_0(t))\bar X_j(t)\Bigr]\dd t \\
		&\quad -\sum_{(i,j,k)\in\mathcal{D}_3^{m}}
		(\bar X_i(t)\cdot\bar X_j(t))\bar X_k(t)\dd t,\qquad 
		\bar X_m(0)=\xi_m.
	\end{aligned}
\end{equation}

Throughout the remainder of this paper, we fix an arbitrary $n \in \mathbb{N}$ as the expansion order.   With these notations, we have the following remainder estimate for the expansion \eqref{eq-dyna-exp}. 
\begin{proposition}[Dynamical fluctuation expansion] \label{thm-dyna-exp-intro} Let  $\varepsilon\in (0,1)$ and assume that $X_\varepsilon$ is the solution of \eqref{SDE-0}  with initial data $\xi_{\varepsilon}$ satisfying Assumption \ref{A1}. Let $(\bar{X}_m)_{0\leq m\leq n}$ be the solution of \eqref{eq-barX0-intro}, \eqref{eq-barX1-intro}, \eqref{eq-barXm-intro}, respectively. We define the remainder terms as $$
w_{\varepsilon,m}(t)
:=\frac{X_\varepsilon(t)-\sum_{k=0}^{m}\varepsilon^{k/2}\bar X_k(t)}{\varepsilon^{m/2}},
\qquad 0\le m\le n,\,\,t>0.
$$ 
Then for each $p\ge1$ and $t>0$, there exists a constant $C(p,t,n)>0$, independent of $\varepsilon$, such that for all $m=0,1,\cdots,n$,
\begin{equation*}
	\mathbb{E}\bigl|w_{\varepsilon,m}(t)\bigr|^{2p}\leq C(p,t,n)\varepsilon^{p}.
\end{equation*}
\end{proposition}
 The proof of Proposition \ref{thm-dyna-exp-intro} yields an estimate uniformly on compact time intervals: for every $T>0$ and $p\ge1$, there exists a constant $C(p,T,n)>0$, independent of $\varepsilon$, such that
\[
\sup_{0\le t\le T}\mathbb{E}\bigl|w_{\varepsilon,m}(t)\bigr|^{2p}
\le C(p,T,n)\varepsilon^p,
\quad 0\le m\le n,
\] and therefore provides a strong dynamical expansion in $L^{2p}$ of $X_\varepsilon$ on finite time intervals. Our second main result is a corresponding weak expansion \eqref{eq-weak-exp-intro} for smooth observables.

\begin{proposition}[Weak expansion]\label{thm-weak-exp-intro}
	Let  $\varepsilon\in (0,1)$, and let $X_\varepsilon$,  $(\bar{X}_m)_{0\leq m\leq n}$ be as in Proposition \ref{thm-dyna-exp-intro}.
    Suppose $F \in C^{n+1}(\mathbb{R}^d)$ satisfies a polynomial growth condition, i.e., there exists $q > 0$ such that \begin{equation}\label{assump-F}
    \max_{0\le i \le n+1} \|D^i F(x)\| \leq C(1+|x|^q), \quad\text{for all } x \in \mathbb{R}^d.
    \end{equation}
We define $a_0(t,F):=\mathbb{E}\big[F(\bar X_0(t))\big]$, and for each $1 \leq m \leq n$,
	\begin{equation}\label{eq-an-intro}
		a_m(t,F):=\mathbb{E}\Bigg[\sum_{i=1}^{m}\frac{1}{i!}	\sum_{(j_1,\dots,j_i)\in\mathcal{D}_i^{m}}
		D^iF\big(\bar X_0(t)\big)\big(\bar X_{j_1}(t),\dots,\bar X_{j_i}(t)\big)
		\Bigg],
	\end{equation}
    where $\mathcal{D}_i^{m}$ is defined in \eqref{D-ni}.
Then for each $t>0$, there exists a constant $C(t,n,F)>0$, independent of $\varepsilon$, such that for all $m=0,1,\cdots,n$,
	\begin{equation*}
		\left|\frac{\mathbb{E}\big[F(X_\varepsilon(t))\big]
		-\sum_{k=0}^m\varepsilon^{k/2}a_k(t,F)}{\varepsilon^{m/2}}\right|\leq C(t, n, F) \varepsilon^{1/2}.
	\end{equation*}
\end{proposition}

Our third main result concerns the long-time behavior of the expansion coefficients and provides answers to the questions \textbf{Q1}-\textbf{Q3} formulated above.

In the one-dimensional case, under a natural non-degeneracy assumption on the leading-order initial data (see Assumption \ref{A2} below), we obtain a complete description of the long-time limits of the coefficients $(a_m(t,F))$ in terms of an explicit recursive relation. Moreover, under the assumption that $\mathbb{P}(\xi_0>0)=\mathbb{P}(\xi_0<0)=1/2$, we show that the long-time limits of these dynamical fluctuation coefficients coincide with the equilibrium fluctuation coefficients, thus providing positive answers to \textbf{Q1} and \textbf{Q2} in the scalar case.

In contrast, in the vector-valued case $d \ge 2$, we construct a smooth observable $F$ for which the second-order dynamical coefficient $a_2(t, F)$ fails to converge as $t \to \infty$. Thus, \textbf{Q1} has a negative answer when $d\geq 2$, and the long-time behavior of the fluctuation expansion is genuinely different in the scalar and vector cases, thereby answering \textbf{Q3}.

When studying the long-time behavior, we further impose the following assumption, which excludes the unstable equilibrium at $x=0$.

\begin{assumption}\label{A2}
There exist constants $ r_{max}> r_{min} > 0$ such that
\begin{equation*}
r_{min} < |\xi_0| < r_{max} \quad \text{almost surely.}
\end{equation*}
\end{assumption}
\begin{theorem}[Long-time behavior in the scalar case]\label{thm-longtime-intro}
	Assume $d=1$. Let $\varepsilon\in(0,1)$ and let $X_\varepsilon$,  $(\bar{X}_m)_{0\leq m\leq n}$,  $(a_m)_{0\leq m\leq n}$ be as in Proposition \ref{thm-weak-exp-intro}. If we further assume that Assumption \ref{A2} holds, then for every $F \in C^{n+1}(\mathbb{R})$ satisfying \eqref{assump-F}, the following assertions hold:\begin{enumerate}

\item[(i)] The limit
$$
b_0(F) := \lim_{t\to\infty} a_0(t,F)
$$
exists, and satisfies
\begin{equation}\label{eq-lim-a0}
b_0(F) = \mathbb{E}\Bigl[ \lim_{t\to\infty} F\bigl(\bar{X}_0(t)\bigr) \Bigr]
= F(1)\mathbb{P}(\xi_0>0) + F(-1)\mathbb{P}(\xi_0<0).
\end{equation}
\item[(ii)] For each $m \in \{1, \dots, n\}$ and $i \in \{1, \dots, m\}$, define
$$
S_{m,i}(t) := \sum_{(j_1,\dots,j_i)\in\mathcal{D}_i^m} \prod_{k=1}^{i}\bar{X}_{j_k}(t),\quad t\ge0.
$$
We adopt the convention that a sum over the empty set is equal to $0$, and hence
$$
S_{m,i}(t):=0
\quad\text{whenever }m<0,\ i<0,\ \text{or } i>m.
$$Then the conditional limits
$$
c_{m,i} := \lim_{t \to \infty}\mathbb{E}[S_{m,i}(t)\mid\xi_0>0],
\quad
\bar{c}_{m,i} := \lim_{t \to \infty}\mathbb{E}[S_{m,i}(t)\mid\xi_0<0]
$$
exist and are uniquely determined by the recursion formula
\begin{equation}\label{eq-c-n-intro}
    c_{m,i} = \frac{i-1}{2}c_{m-2,i-2} - \frac{3}{2}c_{m,i+1} - \frac{1}{2}c_{m,i+2}, 
    \quad 
    \bar c_{m,i} = \frac{i-1}{2}\bar c_{m-2,i-2} + \frac{3}{2}\bar c_{m,i+1} - \frac{1}{2}\bar c_{m,i+2},
\end{equation}
with $c_{1,1}=\bar c_{1,1}=0$, $c_{2,2}=\bar c_{2,2}=\frac12$ and the convention that $c_{0,0}=\bar c_{0,0}=1$, and
		$$c_{k,j}=\bar c_{k,j}=0 \quad\text{whenever } (k,j)\notin \{(0,0)\}\cup\{(k,j): k\ge1,\ 1\le j\le k\}.$$ 
The recursion \eqref{eq-c-n-intro} is understood as a nested induction: one first increases the order $m$, and for each fixed $m$, computes $c_{m,i}$ backward in $i$ from $i=m$ to $1$. The same procedure applies to $\bar c_{m,i}$.

\item[(iii)] For each $m \in \{1, \dots, n\}$, the limit
\begin{equation}\label{eq-b-n}
b_m(F) := \lim_{t \to \infty} a_m(t,F)
\end{equation}
exists and is given by
\begin{equation}\label{eq-b-n-intro}
    b_m(F) = \mathbb{P}(\xi_0>0)\sum_{i=1}^{m} \frac{c_{m,i}}{i!} F^{(i)}(1)
    + \mathbb{P}(\xi_0<0)\sum_{i=1}^{m} \frac{\bar{c}_{m,i}}{i!}F^{(i)}(-1).
\end{equation}
In particular, $b_m(F)=0$ for all odd $m$.
\end{enumerate}
\end{theorem}
Furthermore, each coefficient $a_m(t,F)$ converges exponentially fast to $b_m(F)$, under the following additional assumption on the initial data.
\begin{assumption}\label{A3}
   For each $m\in\{1,2,\cdots,n\}$ and $i\in\{1,2,\cdots,m\}$, there exist a constant $C>0$ depending on $m,i$, such that the conditional expectation
   \begin{equation*}
  \left|\sum_{(j_1,\dots,j_i)\in\mathcal{D}_i^m} \mathbb{E}\left[\prod_{k=1}^{i}\xi_{j_k}\mid \xi_0\right]\right|\leq C,\quad\text{ almost surely.}
   \end{equation*}
\end{assumption}
In particular, Assumption \ref{A3} is implied by Assumption \ref{A1} whenever
$\sigma(\xi_k, k\geq 1)$ is independent of $\sigma(\xi_0)$.
\begin{theorem}[Convergence rates of the long-time limits]\label{thm-rate-intro}
Under the assumptions of Theorem \ref{thm-longtime-intro} and Assumption \ref{A3}, for each $m\in\{1,\dots,n\}$ and $i\in\{1,\dots,m\}$, there exists a constant $C(m,i)>0$ such that, for all $t\ge0$, 
	\begin{equation}\label{eq-cov-rate-S-weak}
	\left|\mathbb E[S_{m,i}(t)\mid\xi_0]-c_{m,i}\right|\le C(m,i)e^{-t},
\quad\text{almost surely on }\{\xi_0>0\},
	\end{equation}
	and
	$$
	\left|\mathbb E[S_{m,i}(t)\mid\xi_0]-\bar c_{m,i}\right|\le C(m,i)e^{-t},
	\quad\text{almost surely on }\{\xi_0<0\}.
	$$
	Moreover, for each $m\in\{0,\dots,n\}$, there exists $C(m,F)>0$ such that, for all $t\ge0$,
	\begin{equation}\label{eq-cov-rate-weak}
	|a_m(t,F)-b_m(F)|\le C(m,F)e^{-t}.
	\end{equation}
\end{theorem}
 Moreover, in this case, we can identify the long-time limits of the dynamical coefficients with the coefficients of the expansion for the invariant measure.
\begin{theorem}[Identification of coefficients, Proposition \ref{prop-invariant-exp}, Theorem \ref{thm-consistency}]\label{thm-eq-vs-inv}
Assume $d=1$. Let $\mu_\varepsilon$ denote the invariant measure of \eqref{SDE-0} and let $(b_m(F))_{0\le m\le n}$ be defined as in Theorem \ref{thm-longtime-intro}. Then for every $F\in C^{n+1}(\mathbb R)$ satisfying \eqref{assump-F}: 
\begin{enumerate}
\item[(i)] There exists $\varepsilon_0\in (0,1]$ and $ (B_m(F))_{0\leq m\leq n} $ such that for every $\varepsilon\in (0,\varepsilon_0)$,
\begin{equation*}
\int_{\mathbb R} F(x)\mu_\varepsilon(\mathrm{d}x)
= \sum_{m=0}^n \varepsilon^{m/2} B_m(F) + R_{n}(F,\varepsilon),
\end{equation*}
where the remainder satisfies
$$
\bigl|R_{n}(F,\varepsilon)\bigr|\le C(n,F)\varepsilon^{(n+1)/2},
$$
for some constant $C(n,F)>0$ independent of $\varepsilon$.
\item[(ii)]  If the initial law of $\xi_0$ satisfies
$$
\mathbb P(\xi_0>0)=\mathbb P(\xi_0<0)=\frac12,
$$
then for every $m\in\{0,\dots,n\}$, we have
$$
B_m(F)=b_m(F).
$$
\end{enumerate}
\end{theorem}
Theorem \ref{thm-longtime-intro}, Theorem \ref{thm-rate-intro} and Theorem \ref{thm-eq-vs-inv} provide positive answers to \textbf{Q1} and \textbf{Q2} in the scalar case $d=1$. 
When $d\ge2$, the long-time behavior of the fluctuation coefficients is different.
\begin{theorem}[Non-convergence in the vector case]\label{thm-vector-case}
Let $d\ge2$ and $n\geq 2$. Then there exists a function $F\in C^{n+1}(\mathbb{R}^d)$ satisfying \eqref{assump-F}, and a family of initial data $\{\xi_\varepsilon\}_{\varepsilon\in(0,1)}$ satisfying Assumption \ref{A1}, Assumption \ref{A2} and Assumption \ref{A3}, such that the second-order coefficient $a_2(t, F)$ (defined as in Proposition \ref{thm-weak-exp-intro}) does not admit a finite limit as $t\to\infty$.
\end{theorem}
 In particular, in dimensions $d\ge2$, one cannot recover the coefficients of the invariant measure expansion from the long-time behavior of the dynamical fluctuation coefficients $(a_m(t,F))_{m\ge0}$. This shows that the answers to \textbf{Q1} and \textbf{Q2} in the scalar case do not extend to higher dimensions, and hence provides an affirmative answer to \textbf{Q3}.

\subsection{Overview of the approach}

The proofs of Propositions \ref{thm-dyna-exp-intro} and \ref{thm-weak-exp-intro} rely on a systematic analysis of the higher-order fluctuation terms. At the formal level, we substitute \eqref{eq-dyna-exp} into \eqref{SDE-0}, expand the nonlinearity $G(x):=|x|^2x$ around the trajectory $\bar X_0$. By matching terms of order $\varepsilon^{m/2}$, we derive the recursive equations \eqref{eq-barX1-intro} \eqref{eq-barXm-intro} for the processes $(\bar X_m)_{1\leq m\leq n}$.
 Since the inhomogeneous term in the equation for $\bar X_m$ only involves the lower-order fluctuations $\bar X_0,\dots,\bar X_{m-1}$, we can derive moment bounds for $(\bar X_m)$ by induction on $m$ (see Proposition \ref{prop-lp-estimate} below). With these moment bounds, Proposition \ref{thm-dyna-exp-intro} is also proved by induction: assuming that the dynamical expansion holds up to order $m-1$, we express the remainder $w_{\varepsilon,m}$ in terms of lower-order terms and use their $L^{2q}$-bounds to control $\|w_{\varepsilon,m}\|_{L^{2p}}$. Proposition \ref{thm-weak-exp-intro} follows from the strong convergence of Proposition \ref{thm-dyna-exp-intro} together with a Taylor expansion of the observable $F$; the coefficients $(a_m(t,F))$ defined in \eqref{eq-an-intro} are precisely the terms arising from this expansion.

The long-time analysis in the scalar case (Theorem \ref{thm-longtime-intro}) requires a more delicate argument based on a nested induction strategy. Since the initial data $\xi_0$ satisfies Assumption \ref{A2}, $\bar X_0(t)$ converges exponentially fast to the stable equilibrium $\operatorname{sgn}(\xi_0)$ as $t \to \infty$, almost surely. Consequently,  $F^{(i)}(\bar X_0(t))$ also converges to $F^{(i)}(\pm 1)$ for each $i\leq n$. 
 Recall that for $m=1,2,\cdots,n$, we can write $$a_m(t,F)= \mathbb{E}\Big[\sum_{i=1}^{m}\frac{1}{i!}	\sum_{(j_1,\dots,j_i)\in\mathcal{D}_i^{m}}
		F^{(i)}\big(\bar X_0\big)S_{m,i}(t)\Big],$$ where $$S_{m, i}(t):=\sum_{(j_1,\dots,j_i)\in\mathcal{D}_i^m} \prod_{k=1}^{i}\bar{X}_{j_k}(t).
$$ 
 The central task is thus to analyze the asymptotic behavior of the terms $S_{m,i}(t)$. By applying It\^o's formula and utilizing combinatorial identities, we find that $S_{m,i}$ satisfies
  \begin{equation}\label{eq-SDE-Smi}
  \begin{aligned}     
\dd S_{m,i}(t)
&= i\sqrt2S_{m-1,i-1}\dd W_t
+ i\bigl(1-3\bar X_0^2\bigr)S_{m,i}\dd t
+i(i-1)S_{m-2,i-2}\dd t\\
&\quad-i(3\bar X_0S_{m,i+1}+S_{m,i+2})\dd t.
  \end{aligned}
   \end{equation}
    The evolution of $S_{m,i}$ is driven by terms with either a lower total order $m-2$ or the same order $m$ but a larger number of factors, namely $S_{m,i+1}$ and $S_{m,i+2}$. This structure leads us to a nested induction: an outer induction on the total order $m$, and for each fixed $m$, an inner induction on the number of factors $i$ that proceeds downwards from $i=m$ to $i=1$; see Figure \ref{fig:nested-induction} for a schematic illustration of the induction scheme. Since $\bar X_0^2(t)$ converges exponentially fast to $1$ as $t \to \infty$, the linear term $i(1 - 3\bar X_0^2(t)) S_{m,i}$ behaves like $-2i S_{m,i}$ for large $t$. By taking expectations and passing to the limit $t \to \infty$, \eqref{eq-SDE-Smi} reduces to the algebraic recursive relation \eqref{eq-c-n-intro} for the limits $c_{m,i}$ and $\bar c_{m,i}$.

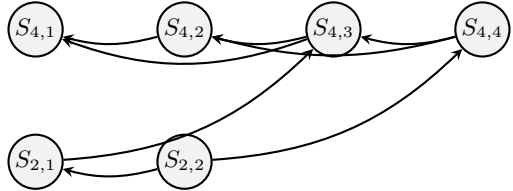
\begin{figure}[ht]
	\centering
	\begin{tikzpicture}[>=stealth,scale=1.1]
		\tikzstyle{mynode}=[circle,draw,thick,fill=gray!10,inner sep=1pt,minimum size=18pt]
		
		% Row n=4: S_{4,i}
		\node[mynode] (S41) at (0.0,0)   {\small $S_{4,1}$};
		\node[mynode] (S42) at (1.8,0)   {\small $S_{4,2}$};
		\node[mynode] (S43) at (3.6,0)   {\small $S_{4,3}$};
		\node[mynode] (S44) at (5.4,0)   {\small $S_{4,4}$};
		
		% Row n=2: S_{2,i}
		\node[mynode] (S21) at (0.0,-1.6) {\small $S_{2,1}$};
		\node[mynode] (S22) at (1.8,-1.6) {\small $S_{2,2}$};
		
		% Arrows
		
		% S_{4,4} uses S_{2,2}
		\draw[->,thick] (S22) to[bend right=20] (S44);
        
		% S_{4,3} uses S_{4,4} and S_{2,1}
		\draw[->,thick] (S44) to[bend left=15] (S43);
		\draw[->,thick] (S21) to[bend right=20] (S43);
		
		% S_{4,2} uses S_{4,3} and S_{4,4}
		\draw[->,thick] (S44) to[bend left=15] (S42);
		\draw[->,thick] (S43) to[bend left=15] (S42);
		
		% S_{4,1} uses S_{4,2} and S_{4,3}
		\draw[->,thick] (S43) to[bend left=20] (S41);
		\draw[->,thick] (S42) to[bend left=15] (S41);
		
		% S_{2,1} uses S_{2,2}
		\draw[->,thick] (S22) to[bend left=15] (S21);
	\end{tikzpicture}
	\caption{Schematic picture of the nested induction on $S_{n,i}$ in the case $n=4$.
	An arrow from $A$ to $B$ indicates that $A$ appears as a forcing term in the evolution equation for $B$.}
	\label{fig:nested-induction}
\end{figure}

The behavior in higher dimensions ($d\ge 2$) is different and stems from the geometric degeneracy of the potential $V$ at its minima.  For $d\ge 2$, the deterministic flow $\bar X_0$ converges to a point on the sphere $\mathbb{S}^{d-1}$, and the linearized operator  $\big(1-|\bar X_0|^2\big)I-2\bar X_0\otimes\bar X_0$  has a zero eigenvalue in the directions tangent to the sphere. In particular, if we write $\hat\xi_0:=\xi_0/|\xi_0|$ and decompose the first-order fluctuation into its radial and tangential components
$$
\bar X_1(t)=r_1(t)\hat\xi_0+v_1(t),
$$
then $r_1$ and $v_1$ satisfy linear SDEs in these eigenspaces, respectively. Along the radial direction, $r_1=\bar X_1 \cdot \hat\xi_0 $ solves an Ornstein-Uhlenbeck type SDE with asymptotically negative drift $$ \left(1-3\left|\bar{X}_0\right|^2\right) r_1\approx -2 r_1,$$ so $\mathbb{E}[r_1(t)^2\mid\xi_0]$ converges to a finite limit as $t\to\infty$. In contrast, along the tangential directions, the drift $(1-|\bar X_0(t)|^2)v_1$ approaches $0$, so $v_1$ behaves asymptotically like a Brownian motion on the tangent space, and
$\mathbb{E}[|v_1(t)|^2\mid\xi_0]\asymp 2(d-1)t$ as $t\to\infty$. At second order, we similarly decompose
$$
\bar X_2(t)=r_2(t)\,\hat\xi_0+v_2(t),
$$
and $r_2=\bar X_2\cdot\hat\xi_0$ satisfies an inhomogeneous linear equation whose forcing terms involve $r_1^2$ and $|v_1|^2$:
$$
\mathrm{d} r_2 =\left(1-3\left|\bar{X}_0\right|^2\right) r_2 \dd t-3\left|\bar{X}_0\right| r_1^2 \dd t -\left|\bar{X}_0\right|\left|v_1\right|^2\dd t.
$$
The $r_1^2$-contribution remains bounded in time, but the forcing coming from $|v_1|^2$ inherits the linear growth of $\mathbb{E}[|v_1(t)|^2\mid\xi_0]$ and yields
$$
\mathbb{E}[r_2(t)\mid\xi_0]\asymp -(d-1)t\quad\text{as }t\to\infty.
$$ This prevents $a_2(t,F)=\mathbb E\big[DF(\bar X_0(t))\bar X_2(t)\big] +\frac12\mathbb E\big[D^2F(\bar X_0(t))(\bar X_1(t),\bar X_1(t))\big]$ from having a finite limit in general.

Finally, we briefly explain how the dynamical coefficients $(b_m(F))$ in Theorem \ref{thm-longtime-intro} can be identified with the coefficients $(B_m(F))$ in \eqref{invariantmeas-expansion} for the invariant measure $\mu_\varepsilon$ of \eqref{SDE-0}, under the condition
 $\mathbb{P}(\xi_0>0)=\mathbb{P}(\xi_0<0)=1/2$. By applying Laplace's method (see, for example, \cite[Chapter 3]{Olv97}) to the density of $\mu_\varepsilon$ around the two non-degenerate minima at $\pm1$, one obtains an expansion of the form
\begin{equation}\label{eq:Bm-structure-intro}
B_m(F) = \frac{1}{2} \sum_{i=0}^{m} \frac{d_{m,i}}{i!} F^{(i)}(1) + \frac{1}{2} \sum_{i=0}^{m} \frac{\bar d_{m,i}}{i!} F^{(i)}(-1), \end{equation}
for certain coefficients $d_{m,i}$ and $\bar d_{m,i}$ that depend only on the potential $V$ and on $m,i$. Comparing \eqref{eq:Bm-structure-intro} with \eqref{eq-b-n-intro} for $b_m(F)$, we see that it is natural to derive recursive relations for $d_{m,i}$ and $\bar d_{m,i}$ and to compare them with \eqref{eq-c-n-intro}. A convenient way to obtain such recursions is to exploit the stationary Fokker--Planck equation satisfied by the invariant density. We can carefully choose test functions for the Fokker--Planck equation and match the coefficients of the different powers of $\varepsilon^{1/2}$. This yields an algebraic recursive relation for the coefficients $(d_{m,i})$ and $(\bar d_{m,i})$, which turns out to coincide with the recursion \eqref{eq-c-n-intro}, with the same initial conditions at low orders. The complete argument is carried out in Section \ref{sec-application}.
 
\subsection{Generalizations}

\

Although our results are formulated for the specific potential function
$V(x) = \frac{1}{4}|x|^4-\frac{1}{2}|x|^2$, the methods in this work are not restricted to this particular choice. More generally, one may consider broader classes of polynomial potentials in the dynamical fluctuation expansions. In that setting, the fluctuation equations and the resulting recursion relations will in general differ from those derived here, and the form and order of the recursion depend on the polynomial structure of $V$. We have chosen to focus on the classical double-well potential because this model provides a particularly transparent setting, in which one can examine in detail the difference between degenerate and non-degenerate structures of the set of minima. 

However, the relation between the long-time dynamical coefficients and the invariant measure expansion coefficients is highly sensitive to the symmetry structure of the set of minima. In the present symmetric double-well case, the two minima have the same depth and their local Taylor expansions are related by symmetry, which leads to a particularly simple balance between the two wells. For a general polynomial potential, such a balance may fail. If the minima do not have the same depth, then only the global minima contribute to the algebraic expansion of the invariant measure, while contributions from higher local minima are exponentially small and hence invisible at the level of polynomial expansions in $\varepsilon^{1/2}$ \cite{Hwa80, AH10}. In this case, matching with the long-time limit of the dynamical expansion coefficients can only be expected when the initial distribution is concentrated in the basins of the global minima. Even when several global minima have the same depth, differences in their Hessians and higher-order derivatives generally lead to different local coefficients and different equilibrium weights, so the identity $b_m=B_m$ requires additional structural conditions.

\subsection{Background and applications}

\ 

\textbf{Background from random dynamical systems.}
The investigation of higher-order fluctuation expansions for SDE systems is not merely of technical interest but is also strongly motivated by questions from the theory of random dynamical systems (RDS). To understand the long-time stability of an RDS, the Lyapunov spectrum quantifies the asymptotic growth rates of infinitesimal perturbations along random trajectories and plays an essential role in the analysis of ergodic properties \cite{Arn98}. In particular, negativity of the top Lyapunov exponent is closely related to asymptotic stability and the synchronization-by-noise phenomenon for RDS \cite{Bax91, FGS17, SV18}. Interestingly, higher-order fluctuation expansions in equilibrium can be viewed as an effective tool for estimating Lyapunov exponents.

A recent breakthrough in this direction was achieved by Gess and Tsatsoulis \cite{GT24}, who derived quantitative estimates for the top Lyapunov exponent of stochastic reaction--diffusion systems, encompassing potentials with both nondegenerate and degenerate minima. More precisely, they consider
\begin{equation}\label{SPDE-0}
\partial_t u_\varepsilon = \Delta u_\varepsilon - \nabla V(u_\varepsilon) + \sqrt{2}\, \varepsilon^{1/2} \xi, \quad \text{on } \mathbb{R}_+ \times \mathbb{T}^1,
\end{equation}
where $V:\mathbb{R}^d \to \mathbb{R}$ is a potential function, $\xi$ denotes the $d$-dimensional space-time white noise, and $\varepsilon>0$ quantifies the intensity of the stochastic perturbation. Under suitable assumptions on $V$, the process $u_\varepsilon$ possesses a unique invariant measure $\mu_\varepsilon$. A central ingredient in their analysis is a small-noise asymptotic expansion of $\mu_\varepsilon$ of the form \eqref{invariantmeas-expansion}, together with careful control of the associated error terms. 

The small-noise expansion approach of \cite{GT24}, however, crucially relies on the availability of an explicit formula for the invariant measure, which is generally unavailable for many stochastic systems. For instance, if the noise in \eqref{SPDE-0} is replaced by a more general trace-class noise, the invariant distribution may no longer admit an explicit Gibbs-type expression \cite[Chapter 11]{DPZ}. Even in finite dimensions, if the additive noise is replaced by a multiplicative one or the drift is non-gradient, the invariant measure typically does not possess a closed-form density \cite{FW84}. Consequently, small-noise asymptotics must be analyzed by alternative techniques; see, for example, \cite{She86, Mik88, BB09, MSNP25}. Due to this limitation, our goal is to develop a complementary approach to learn the expansion coefficients of the invariant measure from the expansion coefficients of the dynamics. We construct an expansion argument that relies solely on the governing equation, without requiring explicit Gibbs-type formulas for the invariant measure.

\textbf{Applications to stochastic optimization and sampling.}
A related perspective comes from stochastic gradient type algorithms arising in both optimization and sampling. In particular, stochastic gradient descent (SGD)
\begin{align}\label{disc-SGD}
\theta_{k+1} = \theta_k - \eta \nabla_{\theta} \hat{L}(\theta_k, \xi_k)
\end{align}
plays a central role in seeking minimizers of a loss function, where $\theta_k\in\mathbb{R}^d$ denotes the parameter at iteration $k$, $\hat{L}(\theta,\xi_k)$ denotes a stochastic approximation of the true loss based on a randomly sampled data point $\xi_k$, and $\eta>0$ represents the learning rate. Under suitable assumptions and scalings, \eqref{disc-SGD} admits continuous-time diffusion approximations; see, for example, \cite{LTE19,SSJ23}.  
In the learning rate dependent SDE surrogate considered in \cite{SSJ23}, the learning rate appears as a scalar prefactor in front of the diffusion term, yielding an additive-noise gradient diffusion of the form \eqref{SDE-0}. Related SDE models also arise in stochastic gradient Langevin dynamics (SGLD), where an isotropic Gaussian perturbation is explicitly injected at each step; see, e.g., \cite{WT11,RRT17}. From this viewpoint, our higher-order expansions may capture non-Gaussian corrections to the accumulated stochastic gradient errors, thereby yielding refined information on the long-term statistical behavior of these stochastic gradient type algorithms.

\subsection{Comments on the literature}

\

\textbf{Small-noise asymptotic expansions for SDEs and SPDEs.}
Small-noise asymptotic expansions for stochastic dynamics arise in several contexts.
At the dynamical level,  Malliavin--Watanabe theory provides power-series type expansions in the small-noise parameter for diffusion processes; see, for example, \cite{Wat87,Yos92,Yos93}.
For SDEs driven by multiplicative L\'evy noise, expansions of the form \eqref{eq-dyna-exp} were obtained in \cite{AS15}.
In the infinite-dimensional setting, small-noise asymptotic expansions for reaction--diffusion type SPDEs with Gaussian and L\'evy noise were constructed in \cite{ADPM11,AMS13}. For singular SPDEs, Friz and Klose \cite{FK22} recently derived Laplace asymptotics for small-noise functionals of the two-dimensional generalized parabolic Anderson model (gPAM), yielding
$$
\mathbb{E}\bigl[\exp\bigl(-F(u_\varepsilon)/\varepsilon\bigr)\bigr]
\sim \exp\Bigl(-\frac{I(F)}{\varepsilon}\Bigr)
\Bigl(c_0(F)+\varepsilon c_1(F)+\cdots\Bigr),
$$
for a broad class of functionals $F$ of the renormalized solution $u_\varepsilon$.
More recently, Gess, the second author and Zhang \cite{GWZ26} established higher-order fluctuation expansions for nonlinear stochastic heat equations in a joint scaling regime combining small noise with a singular limit.

At equilibrium, related perturbative expansions have been obtained for Gibbs-type invariant measures.
Laplace-type expansions for Gaussian functional integrals go back to \cite{ER82,ER81}.
Building on this viewpoint, small-noise expansions for Gibbs-type invariant measures of stochastic reaction-diffusion equations were derived in \cite{GT24}; see also \cite{GST24} for low-temperature expansions of Gibbs measures associated with singular Euclidean fields, such as the $\Phi^4_2$ model.
A complementary strand concerns asymptotics on the exponential (large deviation) scale, leading to Wentzel--Kramers--Brillouin (WKB) type representations for invariant densities $\rho_\varepsilon$.
In the classical Freidlin--Wentzell theory \cite{FW84}, the long-time behavior of small-noise diffusions is governed by a quasipotential $U$, yielding logarithmic asymptotics of the form
$$
\rho_\varepsilon(x)=\exp\Bigl(-\frac{U(x)+o(1)}{\varepsilon}\Bigr),\quad \varepsilon\to0,
$$
i.e. the leading term is specified up to subexponential factors.
Under additional non-degeneracy assumptions, one can refine this to an expansion with both exponential and polynomial contributions, typically written in the WKB form
$$
\rho_\varepsilon(x)\sim \exp\Bigl(-\frac{U(x)}{\varepsilon}\Bigr)
\Bigl(\varphi_0(x)+\varepsilon \varphi_1(x)+\varepsilon^2 \varphi_2(x)+\cdots\Bigr),
\qquad \varepsilon\to0,
$$
where $U$ solves a Hamilton--Jacobi equation and $\varphi_k$ are determined recursively.
For finite-dimensional diffusions, such WKB type expansions of invariant densities were obtained in \cite{She86,Mik88}.

\smallskip
\noindent\textbf{Double-well potentials, metastability and synchronization by noise.}
Double-well and multi-well potentials are canonical models for thermally activated transitions between metastable states, going back to Eyring's transition state theory and Kramers' reaction rate theory \cite{Eyr35,Kra40}.
For reversible overdamped gradient diffusions in multi-well landscapes, sharp Eyring--Kramers asymptotics for exit times and transition rates were established in \cite{BEGK04,BEGK05}; see also \cite{BdH15}.
Beyond the reversible overdamped setting, Eyring--Kramers type formulas for non-reversible metastable diffusions were derived in \cite{LS22}, and for underdamped Langevin dynamics, the corresponding transition time asymptotics were identified in \cite{LRS25}.

Besides metastability, double-well dynamics provides a canonical setting for the study of synchronization by noise. In the one-dimensional case, the drift $x-x^3$ in \eqref{SDE-0} corresponds to the bistable regime of the pitchfork bifurcation family $ax-x^3$ investigated in \cite{CF98,CDLR17}.
For SDEs with additive noise and multidimensional potentials, the negativity of the top Lyapunov exponent and synchronization by noise were studied in \cite{FGS17}, and analogous phenomena for stochastic reaction--diffusion equations with double-well potential were analyzed in \cite{GT24}.

\smallskip
\noindent\textbf{Overdamped versus underdamped Langevin dynamics.}
The overdamped Langevin dynamics \eqref{SDE-0}
is a reversible diffusion with respect to the Gibbs measure proportional to $\exp(-V/\varepsilon)$ and is widely used to model strongly damped motion in a potential landscape. The underdamped (kinetic) Langevin dynamics evolves on phase space $(x,v)\in \mathbb{R}^{2d}$, where $x$ and $v$ denote position and velocity, respectively. Unlike the overdamped case, this system retains inertia, and its evolution is described by:
$$
\dd x_t = v_t\dd t,\qquad
m\mathrm{d}v_t = -\gamma v_t\dd t - \nabla V(x_t)\dd t + \sqrt{2\gamma\varepsilon}\dd W_t,
$$
with invariant measure proportional to $\exp\bigl(-\tfrac{1}{\varepsilon}(V(x)+m|v|^2/2)\bigr)\dd x\dd v$, whose position marginal coincides with the invariant Gibbs measure of the overdamped dynamics. Both formulations play a central role in molecular dynamics and sampling algorithms \cite{LRS10,EGZ19}, and are linked by the small-mass (or large-friction) limit, in which the underdamped process converges to its overdamped counterpart and higher-order corrections in the mass parameter can be quantified \cite{BW20, LLX26}.

\subsection{Structure of the paper}
The rest of this paper is organized as follows. In Section \ref{sec-dyna-expansion}, we establish the dynamical expansion for the overdamped Langevin SDEs. In Section \ref{sec-weak-expansion}, we derive the expansion of the laws with respect to smooth observables. In Section \ref{sec-long-time}, we prove the existence of long-time limits of these expansion coefficients and present their convergence rates. Finally, in Section \ref{sec-application}, we identify these long-time limits, showing that they coincide with the corresponding equilibrium expansion coefficients.

 \section{Dynamical asymptotic expansion}\label{sec-dyna-expansion}
 
 In this section, we prove Proposition \ref{thm-dyna-exp-intro}.
  We first establish $L^{p}$-estimates for the coefficients, and then prove the convergence of the remainder terms.

\subsection{$L^{2p}$-Estimates for the expansion coefficients}

Recall that $\bar X_0$ and $\bar X_1$ solve \eqref{eq-barX0-intro} and \eqref{eq-barX1-intro}, and for $m\ge2$ the coefficient $\bar X_m$ satisfies \eqref{eq-barXm-intro}.
We now derive $ L^{2p} $-estimates for $ (\bar{X}_m)_{0\leq m\leq n} $.
\begin{proposition}\label{prop-lp-estimate}
   Let  $\varepsilon\in (0,1)$ and assume that $X_\varepsilon$ is the solution of \eqref{SDE-0} with initial data satisfying Assumption \ref{A1}. Let $(\bar{X}_m)_{0\leq m\leq n}$  be the solution of \eqref{eq-barX0-intro}-\eqref{eq-barXm-intro}, respectively. 
  Then, for each $p\geq 1$ and $t>0$, there exists a constant $C(p,t,n)>0$, such that for every $m\in\{0,\dots,n\}$,
  \begin{align*}
    	\mathbb{E}\big|\bar{X}_{m}(t)\big|^{2p}\leqslant C(p,t,n).
    \end{align*} 
\end{proposition}

\begin{proof}
    We prove the estimate for $m\in\{0,\dots,n\}$ by induction on $m$.

	\noindent\textbf{Base case $m=0$:}
	Recall that $\bar{X}_0$ is the solution of \eqref{eq-barX0-intro}.
$ \bar{X}_0(t)$ admits the explicit form
\begin{align}\label{eq-explicit-X0t}
	\bar{X}_0(t)=\frac{\xi_0}{\sqrt{|\xi_0|^2+\left(1-|\xi_0|^2\right) e^{-2t}}},
\end{align}
which implies that
\begin{align}\label{eq-bdd-X0t}
\min\{1,|\xi_0|^2\}\leq	|\bar{X}_0(t)|^2\leq \max\{1,|\xi_0|^2\}.
\end{align}
Hence, using Assumption \ref{A1}, we obtain
\begin{equation}\label{eq-Lp-0}
	\mathbb{E}|\bar{X}_{0}(t)|^{2p}\leq 1+\mathbb{E}|\xi_0|^{2p} \leq C(p).
\end{equation}
This completes the proof for $m=0$.

\noindent\textbf{Base case $m=1$:}
$ \bar{X}_1 $ is the solution of \eqref{eq-barX1-intro}.
Applying It\^o's formula to $|\bar{X}_1|^{2p}$, we have
\begin{equation}\label{ito-X1}
	\begin{aligned}
		\frac{1}{2p} \mathrm{d}|\bar{X}_{1}|^{2p}   &=(1-|\bar{X}_{0}|^2)|\bar{X}_{1}|^{2p}\dd t-2(\bar{X}_{0}\cdot \bar{X}_{1})^2|\bar{X}_{1}|^{2p-2}\dd t\\&+\sqrt{2}|\bar{X}_{1}|^{2p-2}\bar{X}_{1}\cdot\dd W_t+(d+2p-2)|\bar{X}_{1}|^{2p-2}\dd t.
	\end{aligned}
\end{equation}
 The term $-2(\bar{X}_{0}\cdot \bar{X}_{1})^2|\bar{X}_{1}|^{2p-2}$ is non-positive, and,
 by Young's inequality, for any $\eta\in(0,1)$,
\begin{equation*}
	(d+2p-2)|\bar{X}_{1}|^{2p-2}\leq C(\eta,p)+\eta|\bar{X}_{1}|^{2p}.
\end{equation*}
Taking expectation in \eqref{ito-X1} and noting that the martingale term vanishes, we obtain 
\begin{equation*}
	\frac{\mathrm{d}}{\mathrm{d}t}\mathbb{E}|\bar{X}_{1}|^{2p}
	\leq C(p)\mathbb{E}|\bar{X}_{1}|^{2p}+C(p).
\end{equation*}
Applying Gronwall's inequality yields
\begin{equation*}
	\mathbb{E}|\bar{X}_{1}(t)|^{2p}\leq C(p,t)\mathbb{E}|\xi_1|^{2p}+C(p,t) .
\end{equation*}

\noindent\textbf{Inductive step:} 
Fix $m\in\{2,\dots,n\}$ and assume that for every $l\in\{0,\dots,m-1\}$ and every $q\ge1$,
\begin{equation}\label{induction-1}
\mathbb{E}|\bar X_l(t)|^{2q}\le C(q,t,n).
\end{equation}
We prove the estimate for $\bar X_m$ with exponent $2p$.

Applying the chain rule to $|\bar{X}_m|^{2p}$ yields
\begin{equation*}
	\begin{aligned}
		\frac{1}{2p} \frac{\mathrm{d}|\bar{X}_{m}|^{2p}}{    \mathrm{d}t}=&(1-|\bar{X}_{0}|^2)|\bar{X}_{m}|^{2p}-2(\bar{X}_{0}\cdot \bar{X}_{m})^2|\bar{X}_{m}|^{2p-2}\\&-\sum_{(i,j)\in \mathcal{D}_2^m}[(\bar{X}_{i}\cdot\bar{X}_{j})(\bar{X}_{0}\cdot\bar{X}_{m})+2(\bar{X}_{i}\cdot\bar{X}_{0})(\bar{X}_{j}\cdot\bar{X}_{m})]|\bar{X}_{m}|^{2p-2}\\&-\sum_{(i,j,k)\in \mathcal{D}_3^m}(\bar{X}_{i}\cdot\bar{X}_{j})(\bar{X}_{k}\cdot\bar{X}_{m})|\bar{X}_{m}|^{2p-2}.
	\end{aligned}
\end{equation*}
We now use Young's inequality to bound the term $(\bar{X}_{i}\cdot\bar{X}_{j})(\bar{X}_{k}\cdot\bar{X}_{m})|\bar{X}_{m}|^{2p-2}$ for instance. For any $\eta\in(0,1)$ and $(i,j,k)\in \mathcal{D}^m_3$,
\begin{equation*}
	\begin{aligned}
\left|(\bar{X}_{i}\cdot\bar{X}_{j})(\bar{X}_{k}\cdot\bar{X}_{m})|\bar{X}_{m}|^{2p-2}\right|&\leq \eta |\bar{X}_{m}|^{2p}+C(\eta,p)\left(|\bar{X}_{i}||\bar{X}_{j}||\bar{X}_{k}|\right)^{2p}\\&\leq \eta |\bar{X}_{m}|^{2p}+C(\eta,p)\left(|\bar{X}_{i}|^{6p}+|\bar{X}_{j}|^{6p}+|\bar{X}_{k}|^{6p}\right).
	\end{aligned}
\end{equation*}
Summing over $i,j,k$ and applying the induction hypothesis \eqref{induction-1}, we arrive at 
\begin{equation*}
	\frac{\mathrm{d}}{\mathrm{d}t}\mathbb{E}|\bar{X}_{m}|^{2p}
	\leq C(p)\mathbb{E}|\bar{X}_{m}|^{2p}+C(p,n).
\end{equation*}
It follows from Gronwall's inequality that
\begin{equation*}
	\mathbb{E}|\bar{X}_{m}(t)|^{2p}\leq C(p,t,n)\mathbb{E}|\xi_m|^{2p}+C(p,t,n) .
\end{equation*}
This completes the inductive step.
\end{proof}

\subsection{Convergence in $L^{2p}$}

Next, we turn to the estimates for the remainder terms. We recall that
 \begin{align*}
w_{\varepsilon, m}:=\frac{X_\varepsilon-\sum_{k=0}^{m}\varepsilon^{k/2}\bar{X}_k}{\varepsilon^{m/2}},\quad 0\leq m\leq n.
\end{align*}

\noindent\textit{Proof of Proposition \ref{thm-dyna-exp-intro}}.
This proposition can also be proved by induction on $m$.
	
 \noindent\textbf{Base case $m=0$:} 
A direct computation shows that $w_{\varepsilon, 0}$ is the solution of
	\begin{equation*}
		\begin{aligned}
		\mathrm{d}w_{\varepsilon, 0}(t)&=\left( w_{\varepsilon, 0}  -|\bar{X}_{0}|^2w_{\varepsilon, 0}  -2(\bar{X}_{0}\cdot w_{\varepsilon, 0})\bar{X}_{0}-2(\bar{X}_{0}\cdot w_{\varepsilon, 0})w_{\varepsilon, 0}-\bar{X}_{0}|w_{\varepsilon, 0}|^2 -|w_{\varepsilon, 0}|^2w_{\varepsilon, 0} \right)\dd t+\sqrt{2\varepsilon}\dd W_t,\\
		w_{\varepsilon, 0}(0)&=\xi_\varepsilon-\xi_0.
			\end{aligned}
\end{equation*}
Applying It\^o's formula to $|w_{\varepsilon, 0}|^{2p}$,  we have
\begin{equation*}
	\begin{aligned}
		\frac{1}{2p} \mathrm{d}|w_{\varepsilon, 0}|^{2p}&=(1-|\bar{X}_{0}|^2)|w_{\varepsilon, 0}|^{2p}\dd t-2(\bar{X}_{0}\cdot w_{\varepsilon, 0})^2|w_{\varepsilon, 0}|^{2p-2}\dd t-3(\bar{X}_{0}\cdot w_{\varepsilon, 0})|w_{\varepsilon, 0}|^{2p}\dd t\\&-|w_{\varepsilon, 0}|^{2p+2}\dd t+\sqrt{2\varepsilon}|w_{\varepsilon, 0}|^{2p-2}w_{\varepsilon, 0}\cdot\dd W_t+\varepsilon(d+2p-2)|w_{\varepsilon, 0}|^{2p-2}\dd t.
	\end{aligned}
\end{equation*}
Taking the expectations yields
\begin{equation}\label{eq-w0-0}
\begin{aligned}
	\frac{\mathrm{d}}{\mathrm{d}t}\mathbb{E}|w_{\varepsilon,0}|^{2p}
	\leq 2p&\mathbb{E}\Big[|w_{\varepsilon,0}|^{2p}-|\bar{X}_{0}|^2|w_{\varepsilon,0}|^{2p}-2(\bar{X}_{0}\cdot w_{\varepsilon, 0})^2|w_{\varepsilon, 0}|^{2p-2}-|w_{\varepsilon,0}|^{2p+2}\\&+3|\bar{X}_{0}\cdot w_{\varepsilon,0}||w_{\varepsilon,0}|^{2p}
	+\varepsilon(d+2p-2)|w_{\varepsilon,0}|^{2p-2}\Big].
\end{aligned}    
\end{equation}
Young's inequality and $|\bar{X}_{0}\cdot w_{\varepsilon,0}|\leq |\bar{X}_{0}|\,| w_{\varepsilon,0}| $ imply that
\begin{equation}\label{eq-w0-1}
\begin{aligned}
3|\bar{X}_{0}\cdot w_{\varepsilon,0}||w_{\varepsilon,0}|^{2p}&\leq 3 (\bar{X}_{0}\cdot w_{\varepsilon, 0})^2|w_{\varepsilon, 0}|^{2p-2} +\frac{3}{4}|w_{\varepsilon,0}|^{2p+2}\\&\leq |\bar{X}_{0}|^2|w_{\varepsilon,0}|^{2p}+2(\bar{X}_{0}\cdot w_{\varepsilon, 0})^2|w_{\varepsilon, 0}|^{2p-2}+\frac{3}{4}|w_{\varepsilon,0}|^{2p+2}.
\end{aligned}    
\end{equation}
	Moreover, for the quadratic variation term,  for any $\eta \in (0,1)$, using Young's inequality again, we obtain
\begin{equation}\label{eq-w0-2}
	(d+2p-2)\varepsilon|w_{\varepsilon,0}|^{2p-2}\leq C(\eta,p)\varepsilon^p+\eta|w_{\varepsilon,0}|^{2p}.
\end{equation}
Therefore, substituting \eqref{eq-w0-1} and \eqref{eq-w0-2} into \eqref{eq-w0-0} yields
\begin{equation*}
	\frac{\mathrm{d}}{\mathrm{d}t}\mathbb{E}|w_{\varepsilon,0}|^{2p}
	\leq 2p(1+\eta)\mathbb{E}|w_{\varepsilon,0}|^{2p}-\frac{p}{2}\mathbb{E}|w_{\varepsilon,0}|^{2p+2}+C(\eta)\varepsilon^p.
\end{equation*}
By Gronwall's inequality and Assumption \ref{A1},
\begin{equation}\label{eq-Lp-w0}
	\mathbb{E}|w_{\varepsilon,0}(t)|^{2p}+\frac{p}{2}\int_{0}^{t}	\mathbb{E}|w_{\varepsilon,0}(s)|^{2p+2}\dd s\leq C(p,t)\left(\mathbb{E}|w_{\varepsilon,0}(0)|^{2p}+\varepsilon^p\right) \leq C(p,t)\varepsilon^p.
\end{equation}

 \noindent\textbf{Base case $m=1$:}  
 $w_{\varepsilon, 1}=\frac{w_{\varepsilon,0}}{\sqrt{\varepsilon}}-\bar{X}_1$ satisfies
\begin{align*}
	\mathrm{d}w_{\varepsilon,1}=\Big( (1-|\bar{X}_{0}|^2)w_{\varepsilon, 1}   &-2(\bar{X}_{0}\cdot w_{\varepsilon, 1})\bar{X}_{0}-2\varepsilon^{-1/2}(\bar{X}_{0}\cdot w_{\varepsilon,0})w_{\varepsilon, 0}   \\&-\varepsilon^{-1/2}|w_{\varepsilon,0}|^2\bar{X}_{0} -\varepsilon^{-1/2}w_{\varepsilon,0}|w_{\varepsilon, 0}|^2\Big)   \dd t.
\end{align*}
For every $ p\geq 1 $, we have
\begin{equation}\label{eq-w1-0}
\begin{aligned}
\frac{1}{2p} \frac{\mathrm{d}|w_{\varepsilon, 1}|^{2p}}{    \mathrm{d}t}=&(1-|\bar{X}_{0}|^2)|w_{\varepsilon, 1}|^{2p}-2(\bar{X}_{0}\cdot w_{\varepsilon, 1})^2|w_{\varepsilon, 1}|^{2p-2}-2\varepsilon^{-1/2} (\bar{X}_{0}\cdot w_{\varepsilon,0})(w_{\varepsilon, 0}\cdot w_{\varepsilon, 1})|w_{\varepsilon, 1}|^{2p-2}\\&-\varepsilon^{-1/2}|w_{\varepsilon, 0}|^{2}(\bar{X}_0\cdot w_{\varepsilon, 1})|w_{\varepsilon, 1}|^{2p-2}- \varepsilon^{-1/2}|w_{\varepsilon, 0}|^{2}(w_{\varepsilon, 0}\cdot w_{\varepsilon, 1})|w_{\varepsilon, 1}|^{2p-2}.
\end{aligned}    
\end{equation}
By Young's inequality,
\begin{equation}\label{eq-w1-1}
\begin{aligned}
2\varepsilon^{-1/2}\left| (\bar{X}_{0}\cdot w_{\varepsilon,0})(w_{\varepsilon, 0}\cdot w_{\varepsilon, 1})|w_{\varepsilon, 1}|^{2p-2}\right|&\leq |w_{\varepsilon, 1}|^{2p}+C \varepsilon^{-p}|w_{\varepsilon, 0}|^{4p}|\bar{X}_0|^{2p}\\&\leq |w_{\varepsilon, 1}|^{2p}+C \varepsilon^{-3p}|w_{\varepsilon, 0}|^{8p}+C\varepsilon^{p}|\bar{X}_0|^{4p}.
\end{aligned}    
\end{equation}
Similarly,
\begin{equation}\label{eq-w1-2}
\begin{aligned}
	\varepsilon^{-1/2}|w_{\varepsilon, 0}|^{2}\left|(\bar{X}_0\cdot w_{\varepsilon, 1})|w_{\varepsilon, 1}|^{2p-2}\right|&\leq |w_{\varepsilon, 1}|^{2p}+C \varepsilon^{-p}|w_{\varepsilon, 0}|^{4p}|\bar{X}_0|^{2p}\\&\leq |w_{\varepsilon, 1}|^{2p}+C \varepsilon^{-3p}|w_{\varepsilon, 0}|^{8p}+C\varepsilon^{p}|\bar{X}_0|^{4p}.
\end{aligned}    
\end{equation}
Finally,
\begin{equation}\label{eq-w1-3}
\begin{aligned}\varepsilon^{-1/2}|w_{\varepsilon, 0}|^{2}|w_{\varepsilon, 0}\cdot w_{\varepsilon, 1}||w_{\varepsilon, 1}|^{2p-2}
	\leq |w_{\varepsilon,1}|^{2p}+C\varepsilon^{-p}|w_{\varepsilon, 0}|^{6p}.
\end{aligned}    
\end{equation}
Substituting \eqref{eq-w1-1}, \eqref{eq-w1-2} and \eqref{eq-w1-3} into \eqref{eq-w1-0},  and using \eqref{eq-Lp-0}, \eqref{eq-Lp-w0}, we deduce that
\begin{align*}
	\frac{\mathrm d}{\mathrm{d} t}\mathbb{E}|w_{\varepsilon,1}|^{2p}
	\leq 4\mathbb{E}|w_{\varepsilon,1}|^{2p}+C(p,t)\varepsilon^p.
\end{align*}
Applying Gronwall's inequality and Assumption \ref{A1}, we obtain that
\begin{align*}
	\mathbb{E}|w_{\varepsilon,1}(t)|^{2p}\leq  C(p,t)\left(\mathbb{E}|w_{\varepsilon,1}(0)|^{2p}+\varepsilon^p\right) \leq C(p,t)\varepsilon^p.
\end{align*}

\noindent\textbf{Inductive step:} 
Fix $m\in\{2,\dots,n\}$ and assume that for every $l\in\{0,\dots,m-1\}$ and every $q\ge1$,
 \begin{equation}\label{eq-wl}
    \mathbb{E}|w_{\varepsilon, l}(t)|^{2q}\leqslant	C(q,t,n)\varepsilon^{q}.
    \end{equation}
Since $m\geq 2$, by expanding the cubic drift term in \eqref{SDE-0} around $\bar X_0$ and arguing as in \cite[Lemma 6.1 and the proof of Theorem 6.5]{GWZ26}, one checks that $w_{\varepsilon, m}$ satisfies
\begin{align*}
	\frac{\mathrm{d}w_{\varepsilon, m}}{\mathrm{d}t}=&(1-|\bar{X}_{0}|^2)w_{\varepsilon, m}  -2(\bar{X}_{0}\cdot w_{\varepsilon, m})\bar{X}_{0}-\sum_{(i,j)\in \mathcal{D}_2^{m}}\Big(2(\bar{X}_{0}\cdot w_{\varepsilon, i})\bar{X}_{j}+(\bar{X}_{j}\cdot w_{\varepsilon, i})\bar{X}_{0}\Big)\\&-2\varepsilon^{-1/2}(\bar{X}_{0}\cdot w_{\varepsilon, 0})w_{\varepsilon,m-1}-\varepsilon^{-1/2}(w_{\varepsilon,m-1}\cdot w_{\varepsilon, 0})\bar{X}_0-\sum_{(i,j,k)\in \mathcal{D}_3^{m}}(w_{\varepsilon, i}\cdot\bar{X}_{j})\bar{X}_k\\&-\sum_{(j,k)\in \mathcal{D}_2^{m}:j>1}\big(\varepsilon^{-1/2}(w_{\varepsilon, 0}\cdot w_{\varepsilon,j-1})\bar{X}_{k}\big)-\varepsilon^{-1}|w_{\varepsilon,0}|^2w_{\varepsilon,m-2},
\end{align*}
 where $\mathcal{D}_2^{m}$ and $\mathcal{D}_3^{m}$ are defined by \eqref{D-ni}.
For every $ p\geq 1 $, we have
\begin{align*}
	\frac{1}{2p} &\frac{\mathrm{d}|w_{\varepsilon, m}|^{2p}}{  \mathrm{d}t}=(1-|\bar{X}_{0}|^2)|w_{\varepsilon, m}|^{2p}-2(\bar{X}_{0}\cdot w_{\varepsilon, m})^2|w_{\varepsilon, m}|^{2p-2}\\&-\sum_{(i,j)\in \mathcal{D}_2^{m}}\big(2 (\bar{X}_{0}\cdot w_{\varepsilon,i})(\bar{X}_j\cdot w_{\varepsilon, m})|w_{\varepsilon, m}|^{2p-2}+(\bar{X}_j\cdot w_{\varepsilon, i})(\bar{X}_0\cdot w_{\varepsilon, m})|w_{\varepsilon, m}|^{2p-2}\big)\\&-
	\varepsilon^{-1/2}\big(2 (\bar{X}_{0}\cdot w_{\varepsilon,0})(w_{\varepsilon, m-1}\cdot w_{\varepsilon, m})|w_{\varepsilon, m}|^{2p-2}+(w_{\varepsilon, m-1}\cdot w_{\varepsilon, 0})(\bar{X}_0\cdot w_{\varepsilon, m})|w_{\varepsilon, m}|^{2p-2}\big)
	\\&-\sum_{(i,j,k)\in \mathcal{D}_3^{m}}(w_{\varepsilon, i}\cdot\bar{X}_{j})(\bar{X}_k\cdot w_{\varepsilon, m})|w_{\varepsilon, m}|^{2p-2}-\varepsilon^{-1/2}\sum_{(j,k)\in \mathcal{D}_2^{m}:j>1}(w_{\varepsilon, 0}\cdot w_{\varepsilon,j-1})(\bar{X}_{k}\cdot w_{\varepsilon, m})|w_{\varepsilon, m}|^{2p-2}\\&- \varepsilon^{-1}|w_{\varepsilon, 0}|^{2}(w_{\varepsilon, m-2}\cdot w_{\varepsilon, m})|w_{\varepsilon, m}|^{2p-2}.
\end{align*}
 The estimates are similar to the case $m=1$. For clarity, we illustrate the argument for the last two terms, as other terms can be handled analogously.
 According to Young's inequality,
 \begin{align*}
 	&\varepsilon^{-1/2}\left| (w_{\varepsilon, 0}\cdot w_{\varepsilon,j-1})(\bar{X}_{k}\cdot w_{\varepsilon, m})\right||w_{\varepsilon, m}|^{2p-2}\\\leq &|w_{\varepsilon, m}|^{2p}+C \varepsilon^{-p}|w_{\varepsilon, 0}|^{2p}|w_{\varepsilon, j-1}|^{2p}|\bar{X}_k|^{2p}\\\leq& |w_{\varepsilon, m}|^{2p}+C \varepsilon^{-3p}|w_{\varepsilon, 0}|^{8p}+C \varepsilon^{-3p}|w_{\varepsilon, j-1}|^{8p}+C\varepsilon^{p}|\bar{X}_k|^{4p}.
 \end{align*}
 Similarly,
 \begin{align*}
 	\varepsilon^{-1}|w_{\varepsilon, 0}|^{2}(w_{\varepsilon, m-2}\cdot w_{\varepsilon, m})|w_{\varepsilon, m}|^{2p-2}\leq& |w_{\varepsilon, m}|^{2p}+C \varepsilon^{-2p}|w_{\varepsilon, 0}|^{4p}|w_{\varepsilon, m-2}|^{2p}\\\leq& |w_{\varepsilon, m}|^{2p}+C \varepsilon^{-3p}|w_{\varepsilon, 0}|^{8p}+C\varepsilon^{-p}|w_{\varepsilon, m-2}|^{4p}.
 \end{align*}
Taking the expectation and using Proposition \ref{prop-lp-estimate} and the induction hypothesis \eqref{eq-wl}, we deduce that
\begin{align*}
	\frac{\mathrm d}{\mathrm{d} t}\mathbb{E}|w_{\varepsilon,m}|^{2p}
	\leq  C(p,n)\mathbb{E}|w_{\varepsilon,m}|^{2p}+C(p,t,n)\varepsilon^p.
\end{align*}
Applying Gronwall's inequality and Assumption \ref{A1}, we conclude that
\begin{align*}
	\mathbb{E}|w_{\varepsilon,m}(t)|^{2p}\leq  C(p,t,n)\left(\mathbb{E}|w_{\varepsilon,m}(0)|^{2p}+\varepsilon^p\right) \leq C(p,t,n)\varepsilon^p.
\end{align*}
\qed

\section{Weak asymptotic expansion}\label{sec-weak-expansion}
In this section, we aim to prove Proposition \ref{thm-weak-exp-intro}.
More precisely, for $m=0,1,\cdots,n$, we will derive an explicit formula for $a_m(t,F)$ and show that the remainder term 
 \begin{align*}
v_m^\varepsilon(t,F):=\frac{	\mathbb{E}\left[F(X_{\varepsilon}(t))\right]-\sum_{k=0}^m\varepsilon^{k/2}a_{k}(t,F)}{\varepsilon^{m/2}}\rightarrow 0, \text{ as } \varepsilon\rightarrow 0.
\end{align*}
 For simplicity, we omit the explicit dependence on $t$ in the notation $X_\varepsilon(t)$, $\bar X_k(t)$, and $a_m(t,F)$, etc.

\noindent\textit{Proof of Proposition \ref{thm-weak-exp-intro}.}
We divide the proof into four steps. First, we identify the coefficients by a Taylor expansion; Then, we estimate the two remainder terms, respectively. Finally, we combine these estimates to obtain the desired conclusion.

\noindent\textbf{Step 1: Taylor expansion and identification of the coefficients.}  We recall that 
 \begin{align*}
w_{\varepsilon, m}=\frac{X_\varepsilon-\sum_{k=0}^{m}\varepsilon^{k/2}\bar{X}_k}{\varepsilon^{m/2}}.
\end{align*}
Hence, \begin{equation}\label{eq:w0-expansion}
w_{\varepsilon,0}
=\sum_{j=1}^m\varepsilon^{j/2}\bar X_j+\varepsilon^{m/2}w_{\varepsilon,m}.
\end{equation}
Since $F\in C^{n+1}(\mathbb{R}^d)$, according to Taylor's theorem at $\bar X_0$, there exists $\zeta_\varepsilon$ between $X_\varepsilon$ and $\bar{X}_0$ such that
\begin{equation}\label{eq:Taylor}
F(X_\varepsilon)
=F(\bar X_0)
+\sum_{i=1}^m\frac{1}{i!}D^iF(\bar X_0)[w_{\varepsilon,0}]^i
+R_m^\varepsilon,
\end{equation}
with $R_m^\varepsilon=\frac{1}{(m+1)!} D^{m+1}F(\zeta_\varepsilon)[w_{\varepsilon,0}]^{m+1}$. By \eqref{assump-F},
\begin{equation}\label{eq:Taylor-R}
|R_m^\varepsilon|
\leq \frac{1}{(m+1)!} \|D^{m+1}F(\zeta_\varepsilon)\| |w_{\varepsilon,0}|^{m+1}\le C\frac{(1+|\bar{X}_0|^q + |X_\varepsilon|^q)}{(m+1)!}|w_{\varepsilon,0}|^{m+1}.
\end{equation}
 For each $i\in\{1,\dots,m\}$, expanding \eqref{eq:w0-expansion} gives
\begin{align*}
	&D^iF(\bar X_0)[w_{\varepsilon,0}]^i
	=
	D^iF(\bar X_0)\Big[\sum_{j=1}^m\varepsilon^{j/2}\bar X_j+\varepsilon^{m/2}w_{\varepsilon,m}\Big]^i\\
	&=
	\sum_{j_1,\dots,j_i=1}^m
	\varepsilon^{(j_1+\cdots+j_i)/2}
	D^iF(\bar X_0)(\bar X_{j_1},\dots,\bar X_{j_i})\\
	&\quad+
	\sum_{l=1}^i \binom{i}{l}
	\sum_{j_1,\dots,j_{i-l}=1}^m
	\varepsilon^{(j_1+\cdots+j_{i-l}+lm)/2}
	D^iF(\bar X_0)\big(
	\bar X_{j_1},\dots,\bar X_{j_{i-l}},
	w_{\varepsilon,m},\dots,w_{\varepsilon,m}
	\big)\\
	&=
	\sum_{k=1}^m \varepsilon^{k/2}
	\sum_{(j_1,\dots,j_i)\in \mathcal D_i^k}
	D^iF(\bar X_0)(\bar X_{j_1},\dots,\bar X_{j_i})+
	\sum_{\substack{j_1,\dots,j_i\in\{1,\dots,m\}\\ j_1+\cdots+j_i\ge m+1}}
	\varepsilon^{(j_1+\cdots+j_i)/2}
	D^iF(\bar X_0)(\bar X_{j_1},\dots,\bar X_{j_i})\\
	&\quad+
	\sum_{l=1}^i \binom{i}{l}
	\sum_{j_1,\dots,j_{i-l}=1}^m
	\varepsilon^{(j_1+\cdots+j_{i-l}+lm)/2}
	D^iF(\bar X_0)\big(
	\bar X_{j_1},\dots,\bar X_{j_{i-l}},
	w_{\varepsilon,m},\dots,w_{\varepsilon,m}
	\big).
\end{align*}
Therefore,
\begin{align*}
	\sum_{i=1}^m \frac1{i!}D^iF(\bar X_0)[w_{\varepsilon,0}]^i
	&=
	\sum_{k=1}^m \varepsilon^{k/2}
	\sum_{i=1}^m \frac1{i!}
	\sum_{(j_1,\dots,j_i)\in \mathcal D_i^k}
	D^iF(\bar X_0)(\bar X_{j_1},\dots,\bar X_{j_i})
	+\widetilde R_m^\varepsilon,
\end{align*}
where
\begin{equation}\label{eq:tilde R}
	\begin{aligned}
	\widetilde R_m^\varepsilon
	&:=
	\sum_{i=1}^m \frac1{i!}
	\sum_{\substack{j_1,\dots,j_i\in\{1,\dots,m\}\\ j_1+\cdots+j_i\ge m+1}}
	\varepsilon^{(j_1+\cdots+j_i)/2}
	D^iF(\bar X_0)(\bar X_{j_1},\dots,\bar X_{j_i})\\
	&\quad+
	\sum_{i=1}^m \frac1{i!}
	\sum_{l=1}^i \binom{i}{l}
	\sum_{j_1,\dots,j_{i-l}=1}^m
	\varepsilon^{(j_1+\cdots+j_{i-l}+lm)/2}
	D^iF(\bar X_0)\big(
	\bar X_{j_1},\dots,\bar X_{j_{i-l}},	w_{\varepsilon,m},\dots,w_{\varepsilon,m}
	\big).
		\end{aligned}
\end{equation}
According to \eqref{eq-an-intro}, we deduce that
\begin{equation}\label{eq:Taylor-structured}
	\mathbb E\left[\sum_{i=1}^m\frac1{i!}D^iF(\bar X_0)[w_{\varepsilon,0}]^i\right]
	=
	\sum_{k=1}^m \varepsilon^{k/2} a_k(F)+\mathbb E\widetilde R_m^\varepsilon.
\end{equation}
Taking expectations in \eqref{eq:Taylor} and using \eqref{eq:Taylor-structured}, we obtain
$$
\mathbb{E} \left[F(X_\varepsilon)\right]
=\sum_{k=0}^m\varepsilon^{k/2}a_k(F)
+\mathbb{E} R_m^\varepsilon+\mathbb{E}\widetilde R_m^\varepsilon,
$$
so that
\begin{equation}\label{eq:vn-expression}
v_m^\varepsilon(F)=
\frac{\mathbb{E} R_m^\varepsilon+\mathbb{E}\widetilde R_m^\varepsilon}{\varepsilon^{m/2}}.
\end{equation}

\medskip\noindent\textbf{Step 2: Estimate of $R_m^\varepsilon$.}
Applying Cauchy-Schwarz inequality to \eqref{eq:Taylor-R}, we obtain
\begin{equation*}
\mathbb{E}|R_m^\varepsilon| \le \frac{C}{(m+1)!} \left( \mathbb{E}\bigl[ (1 + |X_\varepsilon|^q + |\bar{X}_0|^q)^2 \bigr] \right)^{1/2} \left( \mathbb{E}|w_{\varepsilon,0}|^{2(m+1)} \right)^{1/2}.
\end{equation*}
Since $X_\varepsilon=\bar X_0+w_{\varepsilon,0}$, Proposition \ref{prop-lp-estimate} and Proposition \ref{thm-dyna-exp-intro} imply that
$$
\mathbb E\bigl[(1+|X_\varepsilon|^q+|\bar X_0|^q)^2\bigr]\le C(t,n,q).
$$ Moreover, Proposition \ref{thm-dyna-exp-intro} yields $\mathbb{E}|w_{\varepsilon,0}|^{2(m+1)} \le C(t,n) \varepsilon^{m+1}$. Consequently, we arrive at
\begin{equation}\label{eq:Rn-bound}
\mathbb{E}|R_m^\varepsilon|
\le C(t,n,F)\varepsilon^{(m+1)/2},\quad\text{for all }\varepsilon\in(0,1).
\end{equation}

\medskip\noindent\textbf{Step 3: Estimate of $\widetilde R_m^\varepsilon$. }
 By the definition of $\widetilde R_m^\varepsilon$ in \eqref{eq:tilde R}, it suffices to estimate the two types of terms appearing there.

First, consider a term of the form
$$
T_\varepsilon
=
\frac{\varepsilon^{(j_1+\cdots+j_i)/2}}{i!}
D^iF(\bar X_0)(\bar X_{j_1},\dots,\bar X_{j_i}),
$$
where $1\le i\le m$, $j_1,\dots,j_i\in\{1,\dots,m\}$, and
$j_1+\cdots+j_i\ge m+1$.
By \eqref{assump-F} and H\"older's inequality,
$$
\mathbb E|T_\varepsilon|
\le
C\varepsilon^{(j_1+\cdots+j_i)/2}
\Bigl(\mathbb E(1+|\bar X_0|^q)^{i+1}\Bigr)^{1/(i+1)}
\prod_{\ell=1}^i
\bigl(\mathbb E|\bar X_{j_\ell}|^{i+1}\bigr)^{1/(i+1)}.
$$
It follows from Proposition \ref{prop-lp-estimate} that $\mathbb E(1+|\bar X_0|^q)^{i+1}$ and $\mathbb E|\bar X_{j_\ell}|^{i+1}$ are finite and independent of $\varepsilon$. Since
$j_1+\cdots+j_i\ge m+1$, we obtain
$$
\mathbb E|T_\varepsilon|
\le C(t,n,F)\varepsilon^{(m+1)/2}.
$$

Next, consider a term of the form
$$
T_\varepsilon
=
\frac{\varepsilon^{(j_1+\cdots+j_{i-l}+lm)/2}}{i!}
D^iF(\bar X_0)\big(
\bar X_{j_1},\dots,\bar X_{j_{i-l}},
w_{\varepsilon,m},\dots,w_{\varepsilon,m}
\big),
$$
where $1\le l\le i\le m$ and $j_1,\dots,j_{i-l}\in\{1,\dots,m\}$.
Again, using \eqref{assump-F} and H\"older's inequality,
\begin{align*}
	\mathbb E|T_\varepsilon|
	\le
	C\varepsilon^{(j_1+\cdots+j_{i-l}+lm)/2}
	\Bigl(\mathbb E(1+|\bar X_0|^q)^{i+1}\Bigr)^{1/(i+1)} 
	\prod_{r=1}^{i-l}
	\bigl(\mathbb E|\bar X_{j_r}|^{i+1}\bigr)^{1/(i+1)}
	\cdot
	\bigl(\mathbb E|w_{\varepsilon,m}|^{i+1}\bigr)^{l/(i+1)}.
\end{align*}
According to Proposition \ref{prop-lp-estimate}, the moments of $\bar X_{j_r}$ are bounded and independent of $\varepsilon$, while Proposition \ref{thm-dyna-exp-intro} yields
$$
\bigl(\mathbb E|w_{\varepsilon,m}|^{i+1}\bigr)^{1/(i+1)}
\le C\varepsilon^{1/2}.
$$
Hence,
$$
\mathbb E|T_\varepsilon|
\le
C(t,n,F)\varepsilon^{(j_1+\cdots+j_{i-l}+lm)/2}\varepsilon^{l/2}.
$$
Since $l\ge1$ and $j_1+\cdots+j_{i-l}+lm\ge m$, it follows that
$$
\mathbb E|T_\varepsilon|
\le C(t,n,F)\varepsilon^{(m+1)/2}.
$$
Summing over the finitely many terms in $\widetilde R_m^\varepsilon$, we obtain
\begin{equation}\label{eq:Rtilde-bound}
	\mathbb{E}|\widetilde R_m^\varepsilon|
	\le C(t,n,F)\varepsilon^{(m+1)/2},
	\quad\text{for all }\varepsilon\in(0,1).
\end{equation}

\medskip\noindent\textbf{Step 4: Conclusion.}
Combining \eqref{eq:vn-expression}, \eqref{eq:Rn-bound}, and \eqref{eq:Rtilde-bound}, we obtain
$$
|v_m^\varepsilon(F)|
\le C(t,n,F)\varepsilon^{1/2},\quad\text{for all }\varepsilon\in(0,1).
$$
This proves the desired weak expansion.
\qed

\section{Long-time behavior of the fluctuation expansion coefficients}\label{sec-long-time}

In this section, we study the long-time asymptotic behavior of the fluctuation expansion coefficients $(a_m(t,F))$. 
In the scalar-valued case $d=1$, we establish that each coefficient $a_m(t,F)$ converges to a finite limit exponentially fast as $t\to\infty$.
However, when $d>1$, due to the different structure of the potential function $V$, this convergence in general fails.

\subsection{Scalar-valued case}\label{sec-long-time-d=1}

In this subsection, we restrict ourselves to the case $d=1$.

\noindent\textit{Proof of Theorem \ref{thm-longtime-intro}.}
	To prove (i), recall that	$a_0(t,F)=\mathbb E F(\bar X_0(t))$,	where $\bar X_0$ is the solution to \eqref{eq-barX0-intro}. By \eqref{eq-explicit-X0t}, for every $\omega\in\Omega$, $\bar X_0(t,\omega)$ has the same sign as $\xi_0(\omega)$ for all $t\ge0$. Moreover, by Assumption \ref{A2} and \eqref{eq-bdd-X0t}, for all $t\geq 0$, almost surely,
$	\bar X_0(t)$ and $\operatorname{sgn}(\xi_0)$
lie in a deterministic compact set:
\begin{equation}\label{eq-I}
	\bar X_0(t),\operatorname{sgn}(\xi_0)\in I:=[-R_0,-r_0]\cup[r_0,R_0],
\end{equation}
with
$
r_0:=\min\{1,r_{\min}\},\, R_0:=\max\{1,r_{\max}\}
$, where $r_{\min}$ and $r_{\max}$ are from Assumption \ref{A2}. Since $F$ is continuous, the family $\{F(\bar X_0)\}_{t\ge0}$ is bounded. On the other hand, \eqref{eq-explicit-X0t} implies that, almost surely, 
$$
\bar X_0(t)\to \operatorname{sgn}(\xi_0),
\quad t\to\infty.
$$ Therefore, by the dominated convergence theorem,
$$
	\lim_{t\to\infty}a_0(t,F)
=\lim_{t\to\infty}\mathbb EF(\bar X_0(t))
	=\mathbb E\Big[\lim_{t\to\infty}F(\bar X_0(t))\Big]
	=F(1)\mathbb P(\xi_0>0)+F(-1)\mathbb P(\xi_0<0).
$$

For (ii) and (iii), since 
$d=1$, \eqref{eq-an-intro} reduces to
$$
	a_m(F)=\mathbb{E}\left[ \sum_{i=1}^{m}\frac{1}{i!}F^{(i)}(\bar{X}_0)S_{m,i} \right].
$$
The fluctuation coefficients $\bar X_k$ satisfy linear equations with random coefficient $1-3\bar X_0^2$. Define the integrating factor
$$
\Lambda(t):=\exp\left(\int_0^t\bigl(3\bar X_0(s)^2-1\bigr)\dd s\right).
$$
Using
$$
\int_0^t(\bar X_0(s)^2-1)\dd s
=\ln \frac{\xi_0}{\bar X_0(t)}
$$ and \eqref{eq-explicit-X0t},
we obtain
\begin{equation}\label{eq-Lambda}
	\Lambda(t)
	=e^{2t}\frac{\xi_0^3}{\bar X_0(t)^3}
	=e^{2t}\bigl(\xi_0^2+(1-\xi_0^2)e^{-2t}\bigr)^{3/2}.
\end{equation}
Consequently, for every $k\ge1$, almost surely,
\begin{equation}\label{Lambda-asym}
\lim_{t\to\infty}\Lambda(t)^{-k}\int_0^t\Lambda(s)^k\dd s=\frac{1}{2k}.
\end{equation}

We prove (ii) by induction on $m$, and for each fixed $m$, by downward induction on $i$, starting from $i=m$.

\noindent\textbf{Base case $m=1$:}
$S_{1,1}=\bar X_1$ solves
$$
\dd\bar X_1
=\bigl(1-3\bar X_0^2\bigr)\bar X_1\dd t+\sqrt2\dd W_t,
\quad
\bar X_1(0)=\xi_1.
$$
Hence
$$
\bar X_1(t)
=\Lambda(t)^{-1}\left(\xi_1+\int_0^t\Lambda(s)\sqrt2\dd W_s\right).
$$
Since $(W_t)$ is independent of $(\xi_0,\xi_1)$, taking conditional expectation  gives
$$
\mathbb E[\bar X_1(t)\mid\xi_0]
=\Lambda(t)^{-1}\mathbb E[\xi_1\mid\xi_0].
$$
By \eqref{eq-Lambda} and Assumption \ref{A2}, we have $\Lambda(t)^{-1}\to0$ as $t\to\infty$, almost surely. Moreover, by Assumption \ref{A1},
$$
\mathbb E\bigl[|\mathbb E[\xi_1\mid\xi_0]|\bigr]\le \mathbb E|\xi_1|<\infty.
$$Therefore,
$$
\mathbb E[\bar X_1(t)\mid\xi_0>0]\to0,
\quad
\mathbb E[\bar X_1(t)\mid\xi_0<0]\to0.
$$
It follows that
$$
c_{1,1}=\bar c_{1,1}=0.
$$
Since $\bar X_0(t)\in I$, almost surely for every $t>0$, and $F^\prime$ is bounded on $I$, we have
$$
\begin{aligned}
|a_1(t,F)|
=\left|\mathbb E\bigl[F^\prime(\bar X_0(t))\bar X_1(t)\bigr]\right| 
&=\left|\mathbb E\bigl[F^\prime(\bar X_0(t))\mathbb E[\bar X_1(t)\mid\xi_0]\bigr]\right| \\
&\le \sup_{x\in I}|F^\prime(x)|
\mathbb E\Bigl[\bigl|\mathbb E[\bar X_1(t)\mid\xi_0]\bigr|\Bigr].
\end{aligned}
$$
Again by the dominated convergence theorem,
$$
\mathbb E\Bigl[\bigl|\mathbb E[\bar X_1(t)\mid\xi_0]\bigr|\Bigr]\to0.
$$
 Therefore,
$$
b_1(F)=0.
$$

	\medskip\noindent\textbf{Base case $m=2$ and $i=2$:}
			$S_{2,2} = (\bar X_1)^2$ satisfies
		\begin{equation*}
			\dd (\bar X_1)^2	= 2 \bigl(1 - 3 \bar X_0^2\bigr) (\bar X_1)^2 \dd t+ 2 \dd t+ 2 \sqrt{2}\bar X_1 \dd W_t.
		\end{equation*}
		The explicit solution  is
		\begin{equation*}
			(\bar X_1(t))^2	= \Lambda(t)^{-2}	\left(	\xi_1^2
			+ 2 \int_0^t \Lambda(s)^2 \dd s	+ \int_0^t \Lambda(s)^2 2 \sqrt{2} \bar X_1(s) \dd W_s
			\right).
		\end{equation*}
		Taking conditional expectation with respect to $\sigma(\xi_0)$ yields
		$$
		\mathbb E\bigl[(\bar X_1(t))^2\mid \xi_0\bigr]
		=\Lambda(t)^{-2}\left(
		\mathbb E[\xi_1^2\mid \xi_0]
		+2\int_0^t\Lambda(s)^2\dd s
		\right).
		$$
		By \eqref{eq-Lambda}, we have $\Lambda(t)^{-2}\to0$ as $t\to\infty$, almost surely. Similar to the proof of $m=1$, $$\Lambda(t)^{-2}
		\mathbb E[\xi_1^2\mid \xi_0]\to 0,\quad \text{almost surely.}$$  Moreover, by \eqref{Lambda-asym},
$$
\Lambda(t)^{-2}\int_0^t\Lambda(s)^2\dd s\to \frac14 \quad \text{almost surely.}
$$
Therefore,
$$
\mathbb E\bigl[(\bar X_1(t))^2\mid \xi_0>0\bigr]\to \frac12,
\quad
\mathbb E\bigl[(\bar X_1(t))^2\mid \xi_0<0\bigr]\to \frac12.
$$
It follows that
$$
c_{2,2}=\bar c_{2,2}=\frac12.
$$
Since $F^{\prime\prime}(\bar X_0(t))$ is $\sigma(\xi_0)$-measurable, we have
$$
\mathbb E\bigl[F^{\prime\prime}(\bar X_0(t))(\bar X_1(t))^2\mid \xi_0\bigr]
=
F^{\prime\prime}(\bar X_0(t))
\mathbb E\bigl[(\bar X_1(t))^2\mid \xi_0\bigr].
$$
On the event $\{\xi_0>0\}$, we have $\bar X_0(t)\to1$, and on the event $\{\xi_0<0\}$, we have $\bar X_0(t)\to-1$. Since $\bar X_0(t)\in I$ almost surely for all $t\ge0$, and the function $F^{\prime\prime}$ is bounded on $I$, by the dominated convergence theorem,
$$
\mathbb E\bigl[F^{\prime\prime}(\bar X_0(t))(\bar X_1(t))^2\mid \xi_0>0\bigr]
\to \frac12F^{\prime\prime}(1),
$$
and
$$
\mathbb E\bigl[F^{\prime\prime}(\bar X_0(t))(\bar X_1(t))^2\mid \xi_0<0\bigr]
\to \frac12F^{\prime\prime}(-1).
$$

		\noindent\textbf{Inductive step:}
		Let $m \ge 2$ and $i \in \{1, \dots, m\}$. Assume that \eqref{eq-c-n-intro} and \eqref{eq-b-n}  hold for indices $(m^\prime, i^\prime)$ such that $m^\prime < m$ or $m^\prime = m$ with $i^\prime > i$.
		
		We now establish the result for $S_{m,i}$.  By the multi-dimensional It\^o's formula (see, for example, \cite[Theorem 3.3]{RY99}), the dynamics of $S_{m,i}$ is given by
		\begin{align*}
			\dd S_{m,i}(t)
			&=\sum_{(j_1,\dots,j_i)\in\mathcal D_i^m}
			\dd\Bigl(\prod_{k=1}^i\bar X_{j_k}(t)\Bigr)\\
			&=\sum_{(j_1,\dots,j_i)\in\mathcal D_i^m}
			\biggl[
			\sum_{k=1}^i\Bigl(\prod_{a\ne k}\bar X_{j_a}\Bigr)\dd\bar X_{j_k}
			+\sum_{1\le a<b\le i}
			\Bigl(\prod_{p\ne a,b}\bar X_{j_p}\Bigr)\dd\langle\bar X_{j_a},\bar X_{j_b}\rangle
			\biggr].
		\end{align*}
		Substituting the SDEs for $\bar X_{j_k}$ from \eqref{eq-barXm-intro}, we evaluate the contribution of each term:
		\begin{enumerate}
			\item \textbf{Diffusion term:} Since $\dd W_t$ only appears in $\dd \bar X_1$, only indices $j_k=1$ contribute to the martingale part. Removing one index $1$ from an $i$-tuple in $\mathcal{D}_i^m$ leaves an $(i-1)$-tuple in $\mathcal{D}_{i-1}^{m-1}$. Accounting for the $i$ possible positions, we obtain $i \sqrt{2} S_{m-1, i-1} \dd W_t$. 
			\item \textbf{Linear drift term:} Collecting the term $(1-3\bar X_0^2)\bar X_{j_k}$ in $\dd \bar X_{j_k}$ for each $k$ produces $i(1-3\bar X_0^2) S_{m,i} \dd t$.
			\item \textbf{Quadratic variation term:} When $i>2$, the term $\dd \langle \bar X_{j_a}, \bar X_{j_b} \rangle = 2 \delta_{j_a, 1} \delta_{j_b, 1} \dd t$ is non-zero only when $j_a=j_b=1$. Removing this pair from $\mathcal{D}_i^m$ results in an $(i-2)$-tuple in $\mathcal{D}_{i-2}^{m-2}$. There are $\binom{i}{2}$ such pairs, contributing $i(i-1) S_{m-2, i-2} \dd t$.
			 For $i=2$, the quadratic variation term is nonzero only when
$(j_1,j_2)=(1,1)$, that is, when $m=2$, and in this case it contributes $2\dd t$.
			\item \textbf{Higher-order term:} For $j_k\ge2$, the cubic term in $\dd\bar X_{j_k}$ gives
			\begin{equation}\label{eq-high-order}
			-\sum_{(j_1,\dots,j_i)\in\mathcal D_i^m}
			\sum_{k:j_k\ge2}
			\Bigl(\prod_{a\ne k}\bar X_{j_a}\Bigr)
			\Bigl(\sum_{(a,b)\in\mathcal D_2^{j_k}}3\bar X_0\bar X_a\bar X_b+\sum_{(a,b,c)\in\mathcal D_3^{j_k}}\bar X_a\bar X_b\bar X_c\Bigr)\dd t.
			\end{equation}
            For $i\geq 2$, by \eqref{D-ni}, we note that, 
\begin{align*}
S_{m,i+2}=\sum_{(j_1,\dots,j_{i+2})\in\mathcal{D}_{i+2}^m} \left(\prod_{k=1}^{i+2}\bar{X}_{j_k}\right)&=\sum_{l=3}^{m-1}\left(\sum_{(j_1,j_2,j_3)\in\mathcal{D}_{3}^l} \bar{X}_{j_1}\bar{X}_{j_2}\bar{X}_{j_3}\right)\sum_{(j_4,\dots,j_{i+2})\in\mathcal{D}_{i-1}^{m-l}} \prod_{k=4}^{i+2}\bar{X}_{j_k}\\&=\sum_{l=3}^{m-1}S_{l,3}S_{m-l,i-1}=\sum_{l=2}^{m-1}S_{l,3}S_{m-l,i-1}.
\end{align*}Similarly, \begin{align*}
S_{m,i+1}=\sum_{l=2}^{m-1}S_{l,2}S_{m-l,i-1}.
\end{align*}
Hence, writing $l:=j_k$ in \eqref{eq-high-order} and summing over all possible choices yields
			\begin{align*}
				-i\sum_{l=2}^{m-1}
				\Bigl(\sum_{(a,b)\in\mathcal D_2^{l}}3\bar X_0\bar X_a\bar X_b+\sum_{(a,b,c)\in\mathcal D_3^{l}}\bar X_a\bar X_b\bar X_c\Bigr)
				S_{m-l,i-1}\dd t&=-i\sum_{l=2}^{m-1}
				\Bigl(3\bar X_0S_{l,2}+S_{l,3}\Bigr)
				S_{m-l,i-1}\dd t\\&=-i(3\bar X_0S_{m,i+1}+S_{m,i+2})\dd t.
			\end{align*}
		\end{enumerate}
		Collecting all terms, we obtain 
		\begin{equation*}
			\begin{aligned}
				\dd S_{m,i}(t)
				&= i\sqrt2S_{m-1,i-1}\dd W_t
				+ i\bigl(1-3\bar X_0^2\bigr)S_{m,i}\dd t
				+i(i-1)S_{m-2,i-2}\dd t\\
				&\quad-i(3\bar X_0S_{m,i+1}+S_{m,i+2})\dd t.
			\end{aligned}
		\end{equation*}

         For $i=1$, we also have \begin{align*}
       \dd S_{m,1}(t)=\dd \bar{X}_m(t)&=\bigl(1-3\bar X_0^2\bigr)\bar{X}_m\dd t-(3\bar X_0S_{m,2}+S_{m,3})\dd t\\&=i\bigl(1-3\bar X_0^2\bigr)S_{m,i}\dd t-i(3\bar X_0S_{m,i+1}+S_{m,i+2})\dd t.
        \end{align*}
		The mild solution is given by
		\begin{align*}
			S_{m,i}(t)=\Lambda(t)^{-i}\Big(S_{m,i}(0)&+i\sqrt2
			\int_0^t \Lambda(s)^i S_{m-1,i-1}(s)\dd W_s+ i(i-1)\int_0^t \Lambda(s)^i S_{m-2,i-2}(s)\dd s\\& -i\int_0^t \Lambda(s)^i (3\bar X_0S_{m,i+1}+S_{m,i+2})\dd s\Big).
		\end{align*}
		Integrating from $0$ to $t$ and taking conditional expectations given $\xi_0$, we find that the martingale term vanishes and
		\begin{equation}\label{eq-Smi-int}
			\begin{aligned}
				\mathbb E[S_{m,i}(t)\mid\xi_0]
				&= \Lambda(t)^{-i}\mathbb E[S_{m,i}(0)\mid\xi_0]+\Lambda(t)^{-i}\int_0^t 
				\Lambda(s)^i\Bigl[i(i-1)\mathbb E[S_{m-2,i-2}(s)\mid\xi_0]\\
				&
				-i\mathbb E\bigl[3\bar X_0(s)S_{m,i+1}(s)+S_{m,i+2}(s)\mid\xi_0\bigr]\Bigr]\dd s.
			\end{aligned}
		\end{equation}
		For the initial data term, by Assumption \ref{A1}, Assumption \ref{A2}, and \eqref{eq-Lambda}, we have, almost surely, 
		\begin{align}\label{S-0}
			\left|\mathbb{E}\left[ \Lambda(t)^{-i}S_{m,i}(0) \mid\xi_0\right]\right|=	\left| \Lambda(t)^{-i} \right|\left|\mathbb{E}[S_{m,i}(0)\mid \xi_0]\right|\to 0.
		\end{align}
By the induction hypothesis, we have
$$
i(i-1)\mathbb E[S_{m-2,i-2}(s)\mid\xi_0>0]
-i\mathbb E\bigl(3\bar X_0(s)S_{m,i+1}(s)+S_{m,i+2}(s)\mid\xi_0>0\bigr)
$$
converges, as $s\to\infty$, to
$$
i(i-1)c_{m-2,i-2}-i\bigl(3c_{m,i+1}+c_{m,i+2}\bigr).
$$
Similarly, on the event $\{\xi_0<0\}$, we have
$$
i(i-1)\mathbb E[S_{m-2,i-2}(s)\mid\xi_0<0]
-i\mathbb E\bigl[3\bar X_0(s)S_{m,i+1}(s)+S_{m,i+2}(s)\mid\xi_0<0\bigr]
$$
converges to
$$
i(i-1)\bar c_{m-2,i-2}+i\bigl(3\bar c_{m,i+1}-\bar c_{m,i+2}\bigr).
$$
Therefore, using \eqref{eq-Smi-int},\eqref{S-0}, and \eqref{Lambda-asym}
we obtain
$$
c_{m,i}
=\frac{i-1}{2}c_{m-2,i-2}-\frac{3}{2}c_{m,i+1}-\frac{1}{2}c_{m,i+2},
$$
and
$$
\bar c_{m,i}
=\frac{i-1}{2}\bar c_{m-2,i-2}+\frac{3}{2}\bar c_{m,i+1}-\frac{1}{2}\bar c_{m,i+2}.
$$
This proves \eqref{eq-c-n-intro}.
 Next we prove \eqref{eq-b-n}. Since
$$
a_m(t,F)=\mathbb E\left[\sum_{i=1}^m \frac{1}{i!}F^{(i)}(\bar X_0(t))S_{m,i}(t)\right],
$$
it suffices to pass to the limit in each summand. Since the derivatives $F^{(i)}$ are bounded on $I$, by dominated convergence theorem,
$$
\mathbb E\bigl[F^{(i)}(\bar X_0(t))S_{m,i}(t)\mid\xi_0>0\bigr]
\to F^{(i)}(1)c_{m,i},
$$
and
$$
\mathbb E\bigl[F^{(i)}(\bar X_0(t))S_{m,i}(t)\mid\xi_0<0\bigr]
\to F^{(i)}(-1)\bar c_{m,i}.
$$
Hence
$$
b_m(F)
=\mathbb P(\xi_0>0)\sum_{i=1}^m\frac{c_{m,i}}{i!}F^{(i)}(1)
+\mathbb P(\xi_0<0)\sum_{i=1}^m\frac{\bar c_{m,i}}{i!}F^{(i)}(-1).
$$
This proves \eqref{eq-b-n}.
Finally, if $m$ is odd, then by the inductive hypothesis and the recursion \eqref{eq-c-n-intro}, a downward induction on $i$ shows that
$$
c_{m,i}=\bar c_{m,i}=0,\quad 1\le i\le m.
$$
This completes the inductive step. 
\qed

\noindent\textit{Proof of Theorem \ref{thm-rate-intro}.}
\textbf{Step 1: Preliminary estimates.}
It follows from \eqref{eq-explicit-X0t} and Assumption \ref{A2} that there exists
$C>0$ such that
\begin{equation}\label{rate-X0}
	|\bar X_0(t)-\operatorname{sgn}(\xi_0)|\le Ce^{-2t},
	\quad t\ge0, \text{ almost surely}.
\end{equation}
Indeed, by \eqref{eq-explicit-X0t},
$$
\bar X_0(t)
=\operatorname{sgn}(\xi_0)\Bigl(1+\frac{1-\xi_0^2}{\xi_0^2}e^{-2t}\Bigr)^{-1/2},
$$
and Assumption \ref{A2} implies that $|\xi_0|$ is essentially bounded above and below away from $0$.

For each $i\ge1$, by \eqref{eq-Lambda},
\begin{equation}\label{eq-Lambda-theta}
	\Lambda(t)^i
	=e^{2it}\theta_i(t),
	\quad
	\theta_i(t):=\bigl(\xi_0^2+(1-\xi_0^2)e^{-2t}\bigr)^{3i/2}.
\end{equation}
According to Assumption \ref{A2}, there exist deterministic constants $c(i),C(i)>0$ such that
\begin{equation}\label{bdd-theta}
	c(i)\le \theta_i(t)\le C(i),
	\quad t\ge0, \text{ almost surely}.
\end{equation}
Hence,
\begin{equation}\label{eq-Lambda-upper}
	\Lambda(t)^{-i}=e^{-2it}\theta_i(t)^{-1}\le c(i)^{-1}e^{-2it}.
\end{equation}
Moreover, since the function $x\mapsto x^{3i/2}$ is Lipschitz on a compact interval
containing $\xi_0^2+(1-\xi_0^2)e^{-2t}$ and $\xi_0^2$, we have
\begin{equation}\label{eq-theta-rate}
	|\theta_i(t)-|\xi_0|^{3i}|\le C(i)e^{-2t},
	\quad t\ge0, \text{ almost surely}.
\end{equation}
Next, we consider
$$
\left|\Lambda(t)^{-i}\int_0^t\Lambda(s)^i\dd s-\frac1{2i}\right|
\le \left|\theta_i(t)^{-1}e^{-2it}\int_0^t e^{2is}\bigl(\theta_i(s)-|\xi_0|^{3i}\bigr)\dd s\right|
+\left|\frac{|\xi_0|^{3i}}{\theta_i(t)}e^{-2it}\int_0^t e^{2is}\dd s-\frac1{2i}\right|.
$$
Using \eqref{eq-theta-rate}, \eqref{bdd-theta}, and
$$
e^{-2it}\int_0^t e^{2is}\dd s=\frac{1-e^{-2it}}{2i},
$$
we obtain
\begin{equation}\label{eq-Lambda-rate-1}
	\left|\Lambda(t)^{-i}\int_0^t\Lambda(s)^i\dd s-\frac1{2i}\right|
	\le C(i)e^{-2it}\int_0^t e^{2(i-1)s}\dd s
	+\frac1{2i}\left|\frac{|\xi_0|^{3i}}{\theta_i(t)}-1\right|
	+\frac{|\xi_0|^{3i}}{2i\theta_i(t)}e^{-2it}, \text{ almost surely}.
\end{equation}
The second and third terms in \eqref{eq-Lambda-rate-1} are bounded by $C(i)e^{-2t}$ and $C(i)e^{-2it}$ almost surely,
respectively. For the first term, if $i=1$, then
$$
e^{-2t}\int_0^t 1\dd s=te^{-2t}\le Ce^{-t},
$$
while if $i\ge2$, then
$$
e^{-2it}\int_0^t e^{2(i-1)s}\dd s\le C(i)e^{-2t}\le C(i)e^{-t}.
$$
Hence,
\begin{equation}\label{eq-rate-Lambda}
	\left|\Lambda(t)^{-i}\int_0^t\Lambda(s)^i\dd s-\frac1{2i}\right|
	\le C(i)e^{-t},
	\quad t\ge0, \text{ almost surely}.
\end{equation}
Finally, by \eqref{eq-Lambda-theta} and \eqref{bdd-theta},
\begin{equation}\label{eq-kernel}
	\Lambda(t)^{-i}\int_0^t\Lambda(s)^ie^{-s}\dd s
	\le C(i)e^{-2it}\int_0^t e^{(2i-1)s}\dd s
	\le C(i)e^{-t},
	\quad t\ge0, \text{ almost surely}.
\end{equation}

	\medskip\noindent\textbf{Step 2: Proof of \eqref{eq-cov-rate-S-weak}.}
	We prove \eqref{eq-cov-rate-S-weak} by induction on $m$, and for each fixed $m$, by downward induction on $i$.
	
	\medskip\noindent\textbf{Base case $m=1$:}
	Recall that $c_{1,1}=\bar c_{1,1}=0$ and
	$$
	\mathbb E[\bar X_1(t)\mid\xi_0]
	=\Lambda(t)^{-1}\mathbb E[\xi_1\mid\xi_0].
	$$
	According to Assumption \ref{A3}, it follows from \eqref{eq-Lambda-upper} that
	$$
	\left|\mathbb E[\bar X_1(t)\mid\xi_0]\right|
	\le Ce^{-2t},
	\quad t\ge0,  \text{ almost surely}.
	$$
	Hence
	$$
	\left|\mathbb E[\bar X_1(t)\mid\xi_0]-c_{1,1}\right|\le Ce^{-2t}\le Ce^{-t}
	\quad\text{ almost surely on }\{\xi_0>0\},
	$$
	and
	$$
	\left|\mathbb E[\bar X_1(t)\mid\xi_0]-\bar c_{1,1}\right|\le Ce^{-2t}\le Ce^{-t}
	\quad\text{ almost surely on }\{\xi_0<0\}.
	$$
	
	\medskip\noindent\textbf{Base case $m=2$ and $i=2$:}
	Recall that $c_{2,2}=\bar c_{2,2}=\frac12$, and
	$$
	\mathbb E[(\bar X_1(t))^2\mid\xi_0]
	=\Lambda(t)^{-2}\left(
	\mathbb E[\xi_1^2\mid\xi_0]+2\int_0^t\Lambda(s)^2\dd s
	\right).
	$$
	Using Assumption \ref{A3}, \eqref{eq-Lambda-upper} and \eqref{eq-rate-Lambda},  we obtain
	$$
	\left|
	\mathbb E[(\bar X_1(t))^2\mid\xi_0]-\frac12
	\right|
	\le
	\Lambda(t)^{-2}\left|\mathbb E[\xi_1^2\mid\xi_0]\right|
	+
	2\left|
	\Lambda(t)^{-2}\int_0^t\Lambda(s)^2\dd s-\frac14
	\right|
	\le Ce^{-t}, \text{ almost surely}.
	$$
	Therefore,
	$$
	\left|
	\mathbb E[(\bar X_1(t))^2\mid\xi_0]-c_{2,2}
	\right|
	\le Ce^{-t}
	\quad\text{ almost surely on }\{\xi_0>0\},
	$$
	and
	$$
	\left|
	\mathbb E[(\bar X_1(t))^2\mid\xi_0]-\bar c_{2,2}
	\right|
	\le Ce^{-t}
	\quad\text{ almost surely on }\{\xi_0<0\}.
	$$
	
	\medskip\noindent\textbf{Inductive step:}
Let $m \ge 2$ and $i \in \{1, \dots, m\}$. Assume that \eqref{eq-cov-rate-S-weak} hold for indices $(m^\prime, i^\prime)$ such that $m^\prime < m$ or $m^\prime = m$ with $i^\prime > i$. We now establish the result for $(m,i)$. We only prove the estimate on the event $\{\xi_0>0\}$, since the case $\{\xi_0<0\}$ is similar.
	
On	 $\{\xi_0>0\}$, by \eqref{eq-Smi-int},
$$
 \begin{aligned} &\quad \mathbb E[S_{m,i}(t)\mid\xi_0]-c_{m,i}\\ &=\Lambda(t)^{-i}\mathbb E[S_{m,i}(0)\mid\xi_0]+\Lambda(t)^{-i}\int_0^t \Lambda(s)^i\Bigl[i(i-1)\mathbb E[S_{m-2,i-2}(s)\mid\xi_0]\\ & \,\,-i\mathbb E\bigl(3\bar X_0(s)S_{m,i+1}(s)+S_{m,i+2}(s)\mid\xi_0\bigr) -2ic_{m,i}\Bigr]\dd s+2ic_{m,i}\left(\Lambda(t)^{-i}\int_0^t\Lambda(s)^i\dd s-\frac1{2i}\right). 
 \end{aligned}
  $$
  For the initial term, according to \eqref{eq-Lambda-upper} and Assumption \ref{A3}, we have
  $$
  \left|\Lambda(t)^{-i}\mathbb E[S_{m,i}(0)\mid\xi_0]\right|
  \le Ce^{-2it}\le Ce^{-t},\text{ almost surely}.
  $$
Moreover, the induction hypothesis gives
	$$
	\left|\mathbb E[S_{m-2,i-2}(s)\mid\xi_0]-c_{m-2,i-2}\right|
	+\left|\mathbb E[S_{m,i+1}(s)\mid\xi_0]-c_{m,i+1}\right|
	+\left|\mathbb E[S_{m,i+2}(s)\mid\xi_0]-c_{m,i+2}\right|
	\le Ce^{-s},
	$$
	almost surely. Hence,
	$$
	\left|\mathbb E[S_{m,i+1}(s)\mid\xi_0]\right|
	\le |c_{m,i+1}|+Ce^{-s}\le C,
	\quad\text{ almost surely on }\{\xi_0>0\}.
	$$
	Together with \eqref{rate-X0},
$$
\begin{aligned}
	&\Bigl|i(i-1)\mathbb E[S_{m-2,i-2}(s)\mid\xi_0]
	-i\mathbb E\bigl(3\bar X_0(s)S_{m,i+1}(s)+S_{m,i+2}(s)\mid\xi_0\bigr)
	-2ic_{m,i}\Bigr|\\
	&\le i(i-1)\left|\mathbb E[S_{m-2,i-2}(s)\mid\xi_0]-c_{m-2,i-2}\right|+3i\left|\mathbb E[S_{m,i+1}(s)\mid\xi_0]-c_{m,i+1}\right|\\&
	\quad	+i\left|\mathbb E[S_{m,i+2}(s)\mid\xi_0]-c_{m,i+2}\right|
	+3i\mathbb |\bar X_0(s)-1|\left|E\bigl[S_{m,i+1}(s)\mid\xi_0\bigr]\right|\\
	&\le Ce^{-s}, \text{ almost surely}.
\end{aligned}
$$
Therefore, applying \eqref{eq-kernel}, almost surely,
$$
\left|
\Lambda(t)^{-i}\int_0^t
\Lambda(s)^i\Bigl[i(i-1)\mathbb E[S_{m-2,i-2}(s)\mid\xi_0]
-i\mathbb E\bigl(3\bar X_0(s)S_{m,i+1}(s)+S_{m,i+2}(s)\mid\xi_0\bigr)
-2ic_{m,i}\Bigr]\dd s
\right|
\le Ce^{-t}.
$$
Moreover, by \eqref{eq-rate-Lambda},
$$
\left|2ic_{m,i}\left(\Lambda(t)^{-i}\int_0^t\Lambda(s)^i\dd s-\frac1{2i}\right)\right|
\le Ce^{-t}, \text{ almost surely}.
$$
Combining these three terms, we obtain
$$
\left|\mathbb E[S_{m,i}(t)\mid\xi_0]-c_{m,i}\right|\le Ce^{-t},
\quad t\ge0,\text{ almost surely}.
$$
This proves \eqref{eq-cov-rate-S-weak}.

 	\medskip\noindent\textbf{Step 3: Proof of \eqref{eq-cov-rate-weak}.}
For $m=0$, by \eqref{eq-lim-a0},
$$
\begin{aligned}
	|a_0(t,F)-b_0(F)|
	&\le \mathbb P(\xi_0>0)
	\mathbb E\bigl[|F(\bar X_0(t))-F(1)|\mid\xi_0>0\bigr]\\
	&\quad+\mathbb P(\xi_0<0)
	\mathbb E\bigl[|F(\bar X_0(t))-F(-1)|\mid\xi_0<0\bigr].
\end{aligned}
$$
Since \eqref{eq-I} holds and $F^\prime$ is bounded on $I$, we have
$$
|F(\bar X_0(t))-F(\operatorname{sgn}(\xi_0))|
\le C(F)|\bar X_0(t)-\operatorname{sgn}(\xi_0)|.
$$
Thus, by \eqref{rate-X0},
$$
|a_0(t,F)-b_0(F)|\le C(F)e^{-2t}\le C(F)e^{-t},
\quad t\ge0.
$$
For $m\in\{1,\dots,n\}$, by \eqref{eq-an-intro} and \eqref{eq-b-n},
$$
\begin{aligned}
	a_m(t,F)-b_m(F)
	&=\mathbb P(\xi_0>0)\sum_{i=1}^m\frac1{i!}
	\Bigl(
	\mathbb E[F^{(i)}(\bar X_0(t))S_{m,i}(t)\mid\xi_0>0]
	-F^{(i)}(1)c_{m,i}
	\Bigr)\\
	&\quad+\mathbb P(\xi_0<0)\sum_{i=1}^m\frac1{i!}
	\Bigl(
	\mathbb E[F^{(i)}(\bar X_0(t))S_{m,i}(t)\mid\xi_0<0]
	-F^{(i)}(-1)\bar c_{m,i}
	\Bigr).
\end{aligned}
$$
Since $F^{(i)}(\bar X_0(t))$ is $\sigma(\xi_0)$-measurable, we have
$$
\mathbb E[F^{(i)}(\bar X_0(t))S_{m,i}(t)\mid\xi_0>0]
=
\mathbb E\Bigl[F^{(i)}(\bar X_0(t))\mathbb E[S_{m,i}(t)\mid\xi_0]\mid\xi_0>0\Bigr].
$$
Therefore,
$$
\begin{aligned}
	&\left|
	\mathbb E[F^{(i)}(\bar X_0(t))S_{m,i}(t)\mid\xi_0>0]
	-F^{(i)}(1)c_{m,i}
	\right|\\
	&\le
	\mathbb E\Bigl[\left|F^{(i)}(\bar X_0(t))-F^{(i)}(1)\right|
	\left|\mathbb E[S_{m,i}(t)\mid\xi_0]\right|\mid\xi_0>0\Bigr]\\
	&\quad+|F^{(i)}(1)|
	\mathbb E\Bigl[\left|\mathbb E[S_{m,i}(t)\mid\xi_0]-c_{m,i}\right|\mid\xi_0>0\Bigr].
\end{aligned}
$$
Since $F^{(i+1)}$ is bounded on $I$, by \eqref{rate-X0} and the induction hypothesis,
$$
	\mathbb E\Bigl[\left|F^{(i)}(\bar X_0(t))-F^{(i)}(1)\right|
\left|\mathbb E[S_{m,i}(t)\mid\xi_0]\right|\mid\xi_0>0\Bigr]
\le C(F)e^{-2t}\le C(F)e^{-t}.
$$
Together with \eqref{eq-cov-rate-S-weak}, this gives
$$
\left|
\mathbb E[F^{(i)}(\bar X_0(t))S_{m,i}(t)\mid\xi_0>0]
-F^{(i)}(1)c_{m,i}
\right|
\le C(F)e^{-t}.
$$
Similarly,
$$
\left|
\mathbb E[F^{(i)}(\bar X_0(t))S_{m,i}(t)\mid\xi_0<0]
-F^{(i)}(-1)\bar c_{m,i}
\right|
\le C(F)e^{-t}.
$$
Summing over $i=1,\cdots,m$, we conclude that
$$
|a_m(t,F)-b_m(F)|\le C(m,F)e^{-t},
\quad t\ge0.
$$
\qed

\subsection{Vector-valued case}

In this subsection, we consider \eqref{SDE-0} with $d\ge 2$. Unlike in the one-dimensional setting, where each coefficient $a_m(t,F)$ converges to $b_m(F)$ as $t\to\infty$, we show that in higher dimensions, the second-order coefficient $a_2(t,F)$ may fail to admit a finite long-time limit.

\noindent\textit{Proof of Theorem \ref{thm-vector-case}.} 
	We choose
	$$
	\xi_\varepsilon\equiv e_1,\quad \varepsilon\in(0,1),
$$
	where $e_1=(1,0,\dots,0)\in\mathbb R^d$. Then
	$$
	\xi_0=e_1,\quad \xi_k=0,\quad\text{for all }k\ge1,
$$
	so that Assumption \ref{A1} is trivially satisfied. Since $|e_1|=1$, Assumption \ref{A2} also holds.
	Moreover, since $e_1$ is an equilibrium point of the deterministic equation \eqref{eq-barX0-intro},
	we have
$$
	\bar X_0(t)\equiv e_1,\quad t\ge0.
$$
	We decompose $\bar X_1$, $\bar X_2$ as
	$$
	\bar X_1(t)=r_1(t)e_1+v_1(t),\quad
	\bar X_2(t)=r_2(t)e_1+v_2(t),
	$$
where $r_k=e_1\cdot\bar X_k$ and $v_k=\bar X_k-r_k e_1$ for $k=1,2$.
	
	Since $\bar X_0(t)\equiv e_1$, the equation for $\bar X_1$ reduces to
	$$
\dd r_1=-2r_1\dd t+\sqrt2e_1\cdot \dd W_t,\quad r_1(0)=0,
$$
	and
$$
\dd v_1=\sqrt2(I-e_1\otimes e_1)\dd W_t,\quad v_1(0)=0.
$$
	Hence, $r_1$ is a one-dimensional Ornstein-Uhlenbeck process, while $v_1$ is a Brownian motion on the $(d-1)$-dimensional subspace $e_1^\perp:=\operatorname{span}\{e_2,\cdots,e_{d}\}$. In particular,
	\begin{equation}\label{eq:Er1}
	\mathbb E[r_1(t)^2]
	=e^{-4t}\int_0^t 2e^{4s}\dd s
	=\frac{1-e^{-4t}}{2}\to \frac12
	\quad\text{as }t\to\infty,
	\end{equation}
and
	\begin{equation}\label{eq:Ev1}
	\mathbb E|v_1(t)|^2=2(d-1)t.
	\end{equation}
	Next, the equation for $\bar X_2$ yields
$$
	\dd r_2
	=-2r_2\dd t-3r_1^2\dd t-|v_1|^2 \dd t,
	\quad r_2(0)=0.
$$
	Therefore,
\begin{equation}\label{eq:Er2}
	\mathbb E[r_2(t)]
=	-e^{-2t}\int_0^t e^{2s}\Bigl(3\mathbb E[r_1(s)^2]+\mathbb E|v_1(s)|^2\Bigr)\dd s.
\end{equation}
Substituting \eqref{eq:Er1} and  \eqref{eq:Ev1} into \eqref{eq:Er2}, we conclude that
$$
	\mathbb E[r_2(t)]\asymp -(d-1)t,
	\quad t\to\infty.
$$
	Finally, let
$$
	F(x)=x_1=e_1\cdot x.
$$
	Then $F\in C^{n+1}(\mathbb R^d)$ and satisfies \eqref{assump-F}. Moreover,
$$
	DF(x)=e_1,\quad D^2F(x)=0.
$$
	Therefore,
$$
	a_2(t,F)	=
	\mathbb E\bigl[DF(\bar X_0(t))\bar X_2(t)\bigr]
	+\frac12\mathbb E\bigl[D^2F(\bar X_0(t))(\bar X_1(t),\bar X_1(t))\bigr]
	=	\mathbb E[e_1\cdot \bar X_2(t)]
	=	\mathbb E[r_2(t)].
$$
	It follows that
$$
	a_2(t,F)\asymp -(d-1)t
	\quad\text{as }t\to\infty,
$$
	and,  in particular, $a_2(t,F)$ does not admit a finite limit.
\qed

\section{Connection with expansions in equilibrium}\label{sec-application}

In the previous section, we showed that the long-time behavior of the dynamical expansion depends strongly on the dimension. When $d\ge2$, the second-order coefficient $a_2(t,F)$ may fail to admit a finite limit as $t\to\infty$. We now show that, in the scalar-valued case $d=1$, the long-time limits of the dynamical expansion coefficients coincide with the coefficients in the small-noise expansion of the invariant measure.

Let $\mu_\varepsilon$ denote the invariant measure of \eqref{SDE-0}, namely
$$
\mu_\varepsilon(\mathrm{d}x)
=Z_\varepsilon^{-1}\exp\left(-\frac{V(x)}{\varepsilon}\right)\dd x,
$$
where $Z_\varepsilon$ is the normalizing constant. In the scalar-valued case $d=1$, Theorem \ref{thm-longtime-intro} shows that, under Assumption \ref{A1} and Assumption \ref{A2}, the coefficients $a_m(t,F)$ in \eqref{eq-weak-exp-intro} satisfy
$$
\lim_{t\to\infty}a_m(t,F)=b_m(F),\quad 0\le m\le n.
$$
On the other hand, the invariant measure admits a small-noise expansion of the form
\begin{equation*}
	\int_{\mathbb R}F(x)\mu_\varepsilon(\mathrm{d} x)
	=\sum_{m=0}^n \varepsilon^{m/2}B_m(F)+o(\varepsilon^{n/2}),
	\qquad \varepsilon\to0.
\end{equation*}
The main purpose of this section is to show that these two families of coefficients coincide, namely,
$$
b_m(F)=B_m(F),\quad 0\le m\le n.
$$

We begin by characterizing the coefficients in the invariant measure expansion.

\begin{proposition}\label{prop-invariant-exp}
	Let $n\in\mathbb N_+$. {Then for every $F\in C^{n+1}(\mathbb R)$ satisfying \eqref{assump-F}, there exists $\varepsilon_0\in (0,1]$ and a sequence of coefficients $(B_m(F))_{0\le m\le n}$ such that, for every $\varepsilon\in (0,\varepsilon_0)$,
	\begin{equation}\label{eq:inv-expansion}
		\int_{\mathbb R}F(x)\mu_\varepsilon(\mathrm d x)
		=\sum_{m=0}^n \varepsilon^{m/2}B_m(F)+R_{n}(F,\varepsilon),
	\end{equation}
	where
	$$
	|R_{n}(F,\varepsilon)|\le C_n(F)\varepsilon^{(n+1)/2}.
	$$
	Moreover, for each $m\in\{0,1,\dots,n\}$,
		\begin{equation}\label{eq:Bm-structure}
			B_m(F)
			=\frac12\sum_{i=0}^m \frac{d_{m,i}}{i!}F^{(i)}(1)
			+\frac12\sum_{i=0}^m \frac{\bar d_{m,i}}{i!}F^{(i)}(-1).
		\end{equation}
	Here the coefficients $d_{m,i}$ and $\bar d_{m,i}$ are uniquely determined by}
		\begin{equation}\label{eq:rec-d}
			d_{m,i}
			=\frac{i-1}{2}d_{m-2,i-2}-\frac32 d_{m,i+1}-\frac12 d_{m,i+2},
			\quad
			\bar d_{m,i}
			=\frac{i-1}{2}\bar d_{m-2,i-2}+\frac32 \bar d_{m,i+1}-\frac12 \bar d_{m,i+2},
		\end{equation}
		for $m\in\{1,\dots,n\}$ and $i\in\{1,\dots,m\}$, with
$$
d_{0,0}=\bar d_{0,0}=1,\quad d_{m,0}=\bar d_{m,0}=0,\quad m\ge1,
$$
and the convention that
$$
d_{k,j}=\bar d_{k,j}=0
\quad\text{whenever }k<0,\text{ or } k\geq1 \text{ and }j\notin\{0,1,\dots,k\}.
$$
In particular, the recursion gives
$$
d_{1,1}=\bar d_{1,1}=0,\quad d_{2,2}=\bar d_{2,2}=\frac12.
$$
\end{proposition}
\begin{proof}
We divide the proof into 3 steps.

\noindent\textbf{Step 1: Localization near the minima and construction of the expansion.}
	Since
	$V$
	has exactly two non-degenerate minima at $\pm1$, with
	$
	V(1)=V(-1)=-\frac14,
	$
	and $V(x)\to+\infty$ as $|x|\to\infty$, we may choose $0<\delta<1$ and constants
	$c_1,C_1,c_2>0$ depending on $\delta$ such that
	\begin{align}
		V(1)+c_1(x-1)^2 \le V(x) \le V(1)+C_1(x-1)^2,
		&\quad\text{for }|x-1|\le\delta,\label{eq:V-local-1}\\
		V(-1)+c_1(x+1)^2 \le V(x) \le V(-1)+C_1(x+1)^2,
		&\quad\text{for }|x+1|\le\delta,\label{eq:V-local--1}\\
		V(x)\ge V(1)+c_2,
		&\quad\text{for }|x-1|\ge\delta,\ |x+1|\ge\delta.\label{eq:V-away}
	\end{align}	
	Let $\chi_+,\chi_-\in C_c^\infty(\mathbb R)$ satisfy
	$$
	\chi_+(x)=1 \ \text{for }|x-1|\le\delta/2,\quad
	\chi_+(x)=0 \ \text{for }|x-1|\ge\delta,
	$$
	 and $\chi_-(x):=\chi_+(-x)$.
    Classical one-dimensional Laplace asymptotics (see, for example, \cite[Section 3.7]{Olv97}) yield
	\begin{equation}\label{eq:Z-bounds}
		c\sqrt{\varepsilon}\mathrm e^{-V(1)/\varepsilon}
		\le Z_\varepsilon\le
		C\sqrt{\varepsilon}\mathrm e^{-V(1)/\varepsilon},
		\qquad 0<\varepsilon<\varepsilon_0,
	\end{equation}
	for some constants $0<c\le C<\infty$ and $\varepsilon_0>0$.
	For $F\in C^{n+1}(\mathbb R)$ satisfying \eqref{assump-F}, define
	$$
	F_+(x):=\sum_{i=0}^n \frac{F^{(i)}(1)}{i!}\chi_+(x)(x-1)^i,
	\quad
	F_-(x):=\sum_{i=0}^n \frac{F^{(i)}(-1)}{i!}\chi_-(x)(x+1)^i,
	$$
	and set
	$
	H:=F-F_+-F_-.
	$
	Then
	\begin{equation}\label{eq-zero-H}
		H^{(i)}(1)=H^{(i)}(-1)=0,\quad i=0,\dots,n.
	\end{equation}
	For each $i\in\{0,\dots,n\}$, since the function $\chi_+(x)(x-1)^i\in C_c^\infty(\mathbb R)$ and its support is contained in a neighborhood of the non-degenerate minimum $x=1$, in which $x=1$ is the unique critical point, we can apply   \cite[Chapter 3, Theorem 8.1]{Olv97} with $\mu=2, \lambda=i+1$ to obtain an asymptotic expansion for $\int_{\mathbb R}\chi_+(x)(x-1)^i
		\mathrm e^{-V(x)/\varepsilon}\dd x$.
Combining this with the corresponding asymptotics for the partition function $Z_\varepsilon$, there exists $d_{m}(i)\in \mathbb{R}$ such that 
	\begin{equation*}
	\int_{\mathbb R}\chi_+(x)(x-1)^i
	\mu_\varepsilon(\mathrm{d}x)=	\frac{1}{Z_\varepsilon}\int_{\mathbb R}\chi_+(x)(x-1)^i
		\mathrm e^{-V(x)/\varepsilon}\dd x
		=\frac12\sum_{m=i}^n \varepsilon^{m/2}d_{m}(i)+r_{n,i}^+(\varepsilon),
	\end{equation*}
	where
	$$
	|r_{n,i}^+(\varepsilon)|\le C\varepsilon^{(n+1)/2},\quad \varepsilon\in (0,\varepsilon_0).
	$$
	Therefore,
	\begin{align}
		\int_{\mathbb R}F_+(x)\mu_\varepsilon(\mathrm{d}x)
		&=\sum_{i=0}^n\frac{F^{(i)}(1)}{i!}
		\int_{\mathbb R}\chi_+(x)(x-1)^i
	\mu_\varepsilon(\mathrm{d}x)\notag\\
		&=\frac12\sum_{m=0}^n \varepsilon^{m/2}
		\sum_{i=0}^m\frac{d_{m}(i)}{i!}F^{(i)}(1)+R_{n,+}(F,\varepsilon),\label{eq-estimate+}
	\end{align}
with
	$$
	|R_{n,+}(F,\varepsilon)|\le C_n(F)\varepsilon^{(n+1)/2},\quad \varepsilon\in (0,\varepsilon_0).
	$$
	Similarly, let $\bar d_{m}(i)$ denote the corresponding coefficients in the expansion of
	$$
	\int_{\mathbb R}\chi_-(x)(x+1)^i
\mu_\varepsilon(\mathrm{d}x).
	$$
	Then
	\begin{equation}\label{eq-estimate-}
		\int_{\mathbb R}F_-(x)\mu_\varepsilon(\mathrm{d}x)
		=\frac12\sum_{m=0}^n \varepsilon^{m/2}
		\sum_{i=0}^m\frac{\bar d_{m}(i)}{i!}F^{(i)}(-1)+R_{n,-}(F,\varepsilon),
	\end{equation}
	with
	$$
	|R_{n,-}(F,\varepsilon)|\le C_n(F)\varepsilon^{(n+1)/2},\quad \varepsilon\in (0,\varepsilon_0).
	$$
	
	\noindent\textbf{Step 2: Estimate of the remainder term.} To prove (i), it remains to estimate the contribution of $H$. Write
	$$
	\int_{\mathbb R}H(x)\mu_\varepsilon(\mathrm d x)
	=I_1(\varepsilon)+I_{-1}(\varepsilon)+I_{\mathrm{out}}(\varepsilon),
	$$
	where
	$$
	I_1(\varepsilon):=\int_{|x-1|\le\delta}H(x)\mu_\varepsilon(\mathrm{d}x),
	\quad
	I_{-1}(\varepsilon):=\int_{|x+1|\le\delta}H(x)\mu_\varepsilon(\mathrm{d}x),
	$$
	and
	$$
	I_{\mathrm{out}}(\varepsilon)
	:=\int_D H(x)\mu_\varepsilon(\mathrm{d}x),
	\quad
	D:=\{x\in\mathbb R:\ |x-1|>\delta,\ |x+1|>\delta\}.
	$$
 By \eqref{eq-zero-H} and Taylor's theorem, we have
	$$
	H(x)=(x-1)^{n+1}h_1(x)\quad\text{for }|x-1|\le\delta,
	\quad
	H(x)=(x+1)^{n+1}h_{-1}(x)\quad\text{for }|x+1|\le\delta,
	$$
	 where $h_1(x)=\frac{1}{n!}\int_0^1 (1-t)^n
H^{(n+1)}(1+t(x-1))\dd t$, and $h_{-1}(x)=\frac{1}{n!}\int_0^1 (1-t)^n
H^{(n+1)}(-1+t(x+1))\dd t$.
	
	Let $M_1:=\sup_{|x-1|\le\delta}|h_1(x)|<\infty$. Using \eqref{eq:Z-bounds} and \eqref{eq:V-local-1},
	\begin{align*}
		|I_1(\varepsilon)|
		&\le Z_\varepsilon^{-1}\int_{|x-1|\le\delta}|H(x)|\mathrm e^{-V(x)/\varepsilon}\dd x\\
		&\le CZ_\varepsilon^{-1}\int_{|x-1|\le\delta}|x-1|^{n+1}\mathrm e^{-V(x)/\varepsilon}\dd x\\
		&\le C\varepsilon^{-1/2}\int_{|x-1|\le\delta}|x-1|^{n+1}
		\mathrm e^{-c_1(x-1)^2/\varepsilon}\dd x.
	\end{align*}
	With the change of variables $y=(x-1)/\sqrt{\varepsilon}$, we obtain
	\begin{align*}
		|I_1(\varepsilon)|
		&\le C\varepsilon^{-1/2}
		\int_{|y|\le\delta/\sqrt{\varepsilon}}
		(\sqrt{\varepsilon}|y|)^{n+1}\mathrm e^{-c_1y^2}\sqrt{\varepsilon}\dd y\\
		&=C\varepsilon^{(n+1)/2}
		\int_{|y|\le\delta/\sqrt{\varepsilon}}|y|^{n+1}\mathrm e^{-c_1y^2}\dd y
		\le C_n(F)\varepsilon^{(n+1)/2}.
	\end{align*}
	The same argument near $x=-1$ gives
	$$
	|I_{-1}(\varepsilon)|\le C_n(F)\varepsilon^{(n+1)/2}.
	$$
	
We next estimate $I_{\mathrm{out}}(\varepsilon)$. Since $F$ satisfies \eqref{assump-F}, and $F_+,F_-$ are compactly supported, there exist constants $C,q^\prime>0$ such that
\begin{equation}\label{growth-H}
    |H(x)|\le C(1+|x|^{q^\prime}),\quad x\in\mathbb R.
\end{equation}
 We choose $R>0$ sufficiently large such that
\begin{equation}\label{eq-R}
    V(x)\ge V(1)+c_2+x^2,\quad |x|\ge R,
\end{equation}
and recall that $c_2$ comes from \eqref{eq:V-away}.
In the following, we decompose $D$ as
$$
D=(D\cap[-R,R])\cup(D\cap\{|x|>R\}),
$$ and estimate the two parts respectively.
On the compact set $D\cap[-R,R]$, by \eqref{eq:V-away} and \eqref{growth-H} ,
$$
\int_{D\cap[-R,R]} |H(x)|\mathrm e^{-V(x)/\varepsilon}\dd x
\le C(R)\mathrm e^{-(V(1)+c_2)/\varepsilon}.
$$
Therefore, using the lower bound in \eqref{eq:Z-bounds},
\begin{equation}\label{eq-HD1}
Z_\varepsilon^{-1}\int_{D\cap[-R,R]} |H(x)|\mathrm e^{-V(x)/\varepsilon}\dd x
\le C(R)\varepsilon^{-1/2}\mathrm e^{-c_2/\varepsilon}.
\end{equation}
On the set $D\cap\{|x|>R\}$, it follows from \eqref{eq-R} that
\begin{align*}
	\int_{D\cap\{|x|>R\}} |H(x)|\mathrm e^{-V(x)/\varepsilon}\dd x
	&\le C\int_{|x|>R}(1+|x|^{q^\prime})\mathrm e^{-(V(1)+c_2+x^2)/\varepsilon}\dd x\\
	&\le C\mathrm e^{-(V(1)+c_2)/\varepsilon}
	\int_{\mathbb R}(1+|x|^{q^\prime})\mathrm e^{-x^2/\varepsilon}\dd x.
\end{align*}
With the change of variables $x=\sqrt\varepsilon\,y$, we get
$$
\int_{\mathbb R}(1+|x|^{q^\prime})\mathrm e^{-x^2/\varepsilon}\dd x
=\sqrt\varepsilon\int_{\mathbb R}(1+\varepsilon^{q/2}|y|^{q^\prime})\mathrm e^{-y^2}\dd y
\le C(q^\prime)\sqrt\varepsilon.
$$
 Thus,
$$
\int_{D\cap\{|x|>R\}} |H(x)|\mathrm e^{-V(x)/\varepsilon}\dd x
\le C\sqrt\varepsilon\,\mathrm e^{-(V(1)+c_2)/\varepsilon}.
$$
Using again \eqref{eq:Z-bounds}, we obtain
\begin{equation}\label{eq-HD2}
Z_\varepsilon^{-1}\int_{D\cap\{|x|>R\}} |H(x)|\mathrm e^{-V(x)/\varepsilon}\dd x
\le C\,\mathrm e^{-c_2/\varepsilon}.
\end{equation}
Combining the two parts \eqref{eq-HD1} and \eqref{eq-HD2}, we obtain that
$$
|I_{\mathrm{out}}(\varepsilon)|
\le C\varepsilon^{-1/2}\mathrm e^{-c_2/\varepsilon}
\le C_n(F)\varepsilon^{(n+1)/2},
\quad 0<\varepsilon<\varepsilon_0.
$$
	Therefore,
	\begin{equation}\label{eq-estimate-H}
		\left|\int_{\mathbb R}H(x)\mu_\varepsilon(\mathrm{d}x)\right|
		\le C_n(F)\varepsilon^{(n+1)/2}.
	\end{equation}
	Combining \eqref{eq-estimate+}, \eqref{eq-estimate-}, and \eqref{eq-estimate-H}, we conclude that
	$$
	\int_{\mathbb R}F(x)\mu_\varepsilon(\mathrm{d}x)
	=\sum_{m=0}^n \varepsilon^{m/2}B_m(F)+R_{n}(F,\varepsilon),
	$$
	with $B_m(F)$ given by \eqref{eq:Bm-structure} with $d_{m,i}=d_m(i)$, $\bar d_{m,i}=\bar d_m(i)$ and
	$$
	|R_{n}(F,\varepsilon)|\le C_n(F)\varepsilon^{(n+1)/2},\quad 0<\varepsilon< \varepsilon_0.
	$$
	This proves the expansion \eqref{eq:inv-expansion}.

\noindent\textbf{Step 3: Derivation of the recursion.} To derive the recursion for $d_{m,i}$ and $\bar d_{m,i}$, it suffices to consider
$F\in C_c^\infty(\mathbb R)$. For such $F$, the functions
$$
G_1(x):=(x-x^3)F^\prime(x),\quad G_2(x):=F^{\prime\prime}(x)
$$
also satisfy the assumptions of part \textup{(i)}, and hence admit expansions of the form \eqref{eq:inv-expansion}.

Since $\mu_\varepsilon$ is the invariant measure for \eqref{SDE-0}, we have
\begin{equation}\label{eq:inv-identity}
	\int_{\mathbb R} (x-x^3)F^\prime(x)\mu_\varepsilon(\mathrm{d}x)
	+\varepsilon\int_{\mathbb R} F^{\prime\prime}(x)\mu_\varepsilon(\mathrm{d}x)=0.
\end{equation}
Substituting the expansions of $G_1$ and $G_2$ into \eqref{eq:inv-identity}, we obtain
\begin{equation}\label{eq-Bm-identity}
	\sum_{m=0}^{n}\varepsilon^{m/2}\Bigl(B_m(G_1)+B_{m-2}(G_2)\Bigr)+R_n^*(F,\varepsilon)=0,
\end{equation}
where $B_{-1}(G_2)=B_{-2}(G_2):=0$ and
$$
|R_n^*(F,\varepsilon)|\le C_n(F)\varepsilon^{(n+1)/2}.
$$
Since \eqref{eq-Bm-identity} holds for all sufficiently small $\varepsilon>0$, it follows that
\begin{equation}\label{m-vs-m-2}
B_m(G_1)+B_{m-2}(G_2)=0,\quad 0\le m\le n.
\end{equation}
More precisely, we compare the coefficients order by order. Set
$$
A_m:=B_m(G_1)+B_{m-2}(G_2),\quad 0\le m\le n.
$$
Then \eqref{eq-Bm-identity} becomes
$$
\sum_{m=0}^n \varepsilon^{m/2}A_m+R_n^*(F,\varepsilon)=0,
\quad
|R_n^*(F,\varepsilon)|\le C_n(F)\varepsilon^{(n+1)/2}.
$$
Letting $\varepsilon\to0$ yields $A_0=0$. Assuming that $A_0=\cdots=A_{m-1}=0$ for some $1\le m\le n$, we divide the above identity by $\varepsilon^{m/2}$ and let $\varepsilon\to0$. Since
$
\varepsilon^{-m/2}R_n^*(F,\varepsilon)\to0
$, we obtain $A_m=0$. Therefore,
$$
B_m(G_1)+B_{m-2}(G_2)=0,\quad 0\le m\le n.
$$
Set
$
b(x):=x-x^3.
$
Then
$$
b(1)=b(-1)=0,\quad b^\prime(\pm1)=-2,\quad b^{\prime\prime}(1)=-6,\quad b^{\prime\prime}(-1)=6,
$$
$$
b^{(3)}(\pm1)=-6,\quad b^{(l)}(\pm1)=0\quad\text{for all }l\ge4.
$$

We first identify the coefficient of $F^{(i)}(1)$ in \eqref{eq-Bm-identity}, where $1\le i\le m$. By Leibniz's rule,
$$
G_1^{(j)}(1)
=\sum_{\ell=0}^j \binom{j}{l} b^{(l)}(1)F^{(j+1-l)}(1).
$$
Since $b(1)=0$ and $b^{(l)}(1)=0$ for all $l\ge4$, the term $F^{(i)}(1)$ can arise only when
$$
j=i,\quad j=i+1,\quad\text{or}\quad j=i+2.
$$
More precisely, its coefficient in $G_1^{(j)}(1)$ is
$$
\binom{i}{1}b^\prime(1)=-2i,\quad
\binom{i+1}{2}b^{\prime\prime}(1)=-3i(i+1),\quad
\binom{i+2}{3}b^{(3)}(1)=-(i+2)(i+1)i.
$$
Therefore, by \eqref{eq:Bm-structure}, the coefficient of $F^{(i)}(1)$ in $B_m(G_1)$ is
$$
-\frac{i}{i!}d_{m,i}
-\frac{3i}{2\,i!}d_{m,i+1}
-\frac{i}{2\,i!}d_{m,i+2}.
$$
On the other hand,
$$
G_2^{(j)}(1)=F^{(j+2)}(1),
$$
so the coefficient of $F^{(i)}(1)$ in $B_{m-2}(G_2)$ is
$$
\frac{1}{2(i-2)!}d_{m-2,i-2}
=\frac{i(i-1)}{2\,i!}d_{m-2,i-2},
$$
with the convention that this term is zero when $i=1$.

Hence, the coefficient of $F^{(i)}(1)$ in
$
B_m(G_1)+B_{m-2}(G_2)
$ 
is
$$
-\frac{i}{i!}d_{m,i}
-\frac{3i}{2i!}d_{m,i+1}
-\frac{i}{2i!}d_{m,i+2}
+\frac{i(i-1)}{2i!}d_{m-2,i-2}.
$$
Since \eqref{m-vs-m-2} holds for arbitrary $F$, the coefficient of $F^{(i)}(1)$ must vanish. We obtain
$$
-2d_{m,i}-3d_{m,i+1}-d_{m,i+2}+(i-1)d_{m-2,i-2}=0,
$$
that is,
$$
d_{m,i}
=\frac{i-1}{2}d_{m-2,i-2}-\frac32 d_{m,i+1}-\frac12 d_{m,i+2}.
$$
The argument at $x=-1$ is identical, except that
$ 
b^{\prime\prime}(-1)=6.
$ 
Thus, we obtain
$$
\bar d_{m,i}
=\frac{i-1}{2}\bar d_{m-2,i-2}+\frac32 \bar d_{m,i+1}-\frac12 \bar d_{m,i+2}.
$$
This proves the recursion \eqref{eq:rec-d}.

It remains to verify the initial values. 
 Since $V$ is even and $\chi_-(x)=\chi_+(-x)$, 
$$
\int_{\mathbb R}\chi_+(x)\mu_\varepsilon(\mathrm{d}x)
=
\int_{\mathbb R}\chi_-(x)\mu_\varepsilon(\mathrm{d}x).
$$
Hence,
$$
d_{m,0}=\bar d_{m,0},\quad 0\le m\le n.
$$
Taking $F\equiv1$ in \eqref{eq:inv-expansion}, we have
$$
1=\sum_{m=0}^n \varepsilon^{m/2}B_m(1)+R_{n,1}(\varepsilon).
$$
By \eqref{eq:Bm-structure},
$$
B_m(1)=\frac12d_{m,0}+\frac12\bar d_{m,0}=d_{m,0}.
$$
Therefore,
$$
d_{0,0}=\bar d_{0,0}=1,\qquad d_{m,0}=\bar d_{m,0}=0\quad\text{for }m\ge1.
$$
Similarly, by choosing $F(x)=x$, we get $d_{1,1}=\bar{d}_{1,1}=0$.

For each fixed $m\geq 1$, the recursion is determined successively by downward induction on $i$. Hence, the families $(d_{m,i})$ and $(\bar d_{m,i})$ are uniquely determined.
\end{proof}

\begin{theorem}\label{thm-consistency} Suppose that $F\in C^{n+1}(\mathbb{R})$ and satisfies \eqref{assump-F}.
	Let $(b_m(F))_{0\leq m\leq n}$ be the long-time limit coefficients defined in Theorem \ref{thm-longtime-intro} with the assumption $\mathbb{P}(\xi_0 > 0) = \mathbb{P}(\xi_0 < 0) = 1/2$. Let $(B_m(F))_{0\leq m\leq n}$  be the coefficients defined in Proposition \ref{prop-invariant-exp}. Then for every $m \in \{0, \dots, n\}$, we have
	\begin{equation*}
		b_m(F) = B_m(F).
	\end{equation*}
\end{theorem}
\begin{proof}
	For $m=0$, the conclusion follows immediately from \eqref{eq-lim-a0}, \eqref{eq:Bm-structure},
	$d_{0,0}=\bar d_{0,0}=1$, and the assumption
	$$
	\mathbb P(\xi_0>0)=\mathbb P(\xi_0<0)=\frac12.
	$$
	Indeed,
	$$
	b_0(F)=\frac12F(1)+\frac12F(-1)=B_0(F).
	$$
	
	Now let $m\in\{1,\dots,n\}$. By Theorem \ref{thm-longtime-intro}, the coefficients
	$c_{m,i}$ and $\bar c_{m,i}$ satisfy
	$$
	c_{m,i}
	=\frac{i-1}{2}c_{m-2,i-2}-\frac32 c_{m,i+1}-\frac12 c_{m,i+2},
	$$
	$$
	\bar c_{m,i}
	=\frac{i-1}{2}\bar c_{m-2,i-2}+\frac32 \bar c_{m,i+1}-\frac12 \bar c_{m,i+2},
	$$
	for $1\le i\le m$, with
	$
	c_{0,0}=\bar c_{0,0}=1,\quad c_{1,1}=\bar c_{1,1}=0
	$,
	and the convention
	$$
	c_{k,j}=\bar c_{k,j}=0
	\quad\text{whenever }k<0,\text{ or } k\geq 1 \text{ and }j\notin\{0,1,\dots,k\}.
	$$
	On the other hand, by Proposition \ref{prop-invariant-exp}, the coefficients
	$d_{m,i}$ and $\bar d_{m,i}$ satisfy exactly the same recursion with the same initial values.
Hence,
	$$
	c_{m,i}=d_{m,i},\quad \bar c_{m,i}=\bar d_{m,i},
	\quad 1\le i\le m,\ 1\le m\le n.
	$$

	Finally, by \eqref{eq-b-n} and \eqref{eq:Bm-structure}, together with
	$$
	\mathbb P(\xi_0>0)=\mathbb P(\xi_0<0)=\frac12,
	$$
	we obtain for every $m\in\{1,\dots,n\}$,
	\begin{align*}
		b_m(F)
		&=\frac12\sum_{i=1}^m \frac{c_{m,i}}{i!}F^{(i)}(1)
		+\frac12\sum_{i=1}^m \frac{\bar c_{m,i}}{i!}F^{(i)}(-1)\\
		&=\frac12\sum_{i=1}^m \frac{d_{m,i}}{i!}F^{(i)}(1)
		+\frac12\sum_{i=1}^m \frac{\bar d_{m,i}}{i!}F^{(i)}(-1)
		=B_m(F).
	\end{align*}
	This completes the proof.
\end{proof}

\section*{Acknowledgements}
The first author acknowledges the support by National Natural Science Foundation of China (No. 12471138, 12571171). 
The second author acknowledges the support by the US Army Research Office, grant W911NF2310230. 

The authors would like to thank Jungkyoung Lee and Jae-Hwan Choi for helpful discussions.

		\bibliographystyle{alpha}
	\bibliography{Expansion}
\end{document}